\documentclass[11pt]{amsart}
\usepackage[margin=1in]{geometry}
\usepackage{amssymb,amsmath,amsthm}
\usepackage{tikz-cd}
\usetikzlibrary{matrix,arrows,decorations.pathmorphing}
\usepackage{hyperref}
\hypersetup{breaklinks, colorlinks}

\newtheorem{thm}{Theorem}[section]
\newtheorem{lemma}[thm]{Lemma}
\newtheorem{cor}[thm]{Corollary}
\newtheorem{prop}[thm]{Proposition}

\theoremstyle{definition}
\newtheorem{defn}[thm]{Definition}

\theoremstyle{remark}
\newtheorem{rem}[thm]{Remark}

\def\O{\mathcal{O}}
\def\A{\mathcal{A}}
\def\B{\mathcal{B}}
\def\G{\mathcal{G}}
\def\L{\mathcal{L}}
\def\M{\mathcal{M}}
\def\P{\mathcal{P}}
\def\Q{\mathcal{Q}}
\def\R{\mathcal{R}}
\def\p{\mathbb{P}}
\def\Ext{\mathrm{Ext}}
\def\Hom{\mathrm{Hom}}
\def\RHom{\mathrm{RHom}}
\def\Gal{\mathrm{Gal}}
\def\Gr{\mathrm{Gr}}
\def\Sym{\mathrm{Sym}}
\newcommand{\wt}{\widetilde}

\numberwithin{equation}{section}

\title{Derived Categories of Quintic Del Pezzo Fibrations}
\author{Fei Xie}
\address{School of Mathematics, University of Edinburgh, James Clerk Maxwell Building, Peter Guthrie Tait Road, Edinburgh, EH9 3FD, UK}
\email{fei.xie@ed.ac.uk}
\begin{document}
\maketitle
%-------------Abstract--------------
\begin{abstract}
We provide a semiorthogonal decomposition for the derived category of fibrations of quintic del Pezzo surfaces with rational Gorenstein singularities. There are three components, two of which are equivalent to the derived categories of the base and the remaining non-trivial component is equivalent to the derived category of a flat and finite of degree 5 scheme over the base. We introduce two methods for the construction of the decomposition. One is the moduli space approach following the work of Kuznetsov on the sextic del Pezzo fibrations and the components are given by the derived categories of fine relative moduli spaces. The other approach is that one can realize the fibration as a linear section of a Grassmannian bundle and apply Homological Projective Duality.
\end{abstract}
%--------------Introduction--------
\section{Introduction}
In this paper, we study the structure of the bounded derived category of coherent sheaves of fibrations of quintic del Pezzo surfaces with rational Gorenstein singularities (equivalently, the minimal resolution of the surface is crepant). The aim is to find a semiorthogonal decomposition for the derived category. The paper is inspired by and follows the strategy of the work of Kuznetsov on fibrations of sextic del Pezzo surfaces \cite{kuzdp6}. 

Other families of del Pezzo surfaces that have been investigated are the cases of degree $9$ \cite{bersb}, of degree $8$ and $4$ \cite{kuzqfib}\cite{abbqfib} (note that a rational Gorenstein quartic del Pezzo surface over any field is a complete intersection of two quadrics in $\p^4$ \cite[Theorem 4.4(i)]{hwgdp} and rational Gorenstein singularity implies smoothness in degree $9$). A rational Gorenstein del Pezzo surface of degree $7$ is the blow-up of $\p^2$ or a nodal quadric \cite[Theorem 29.4]{cubicforms}\cite[Proposition 8.1]{ctssdp}. Hence, the case of degree $7$ can be reduced to higher degrees. Our work on the quintic case will complete the picture of del Pezzo fibrations of degree at least $4$.

Let $\mathcal{X}\to S$ be a flat family of del Pezzo surfaces over a smooth variety $S$. In aforementioned examples (degree $4,6,8,9$), there is an $S$-linear semiorthogonal decomposition of the type
\[D^b(\mathcal{X})=\langle\A_1,\dots,\A_d\rangle\]
where $d=3$ for degree $6,8,9$ and $d=2$ for degree $4$. The subcategory $\A_i$ is equivalent to the derived category of an algebraic variety $D^b(Z_i)$ or an Azumaya variety $D^b(Z_i, B_i)$ where $Z_i$ is flat over $S$ and in the cases of degree $6,8,9$, also finite over $S$. Here $B_i$ is a sheaf of Azumaya algebras over $Z_i$ and $D^b(Z_i,B_i)$ is the derived category of coherent sheaves of modules over $B_i$ or equivalently the derived category of $\beta_i$-twisted ($\beta_i$ represents the Brauer class of the algebra $B_i$) coherent sheaves on $Z_i$.

When $\mathcal{X}$ is a del Pezzo surface $X$ over a field $k$, that is, $S=\mathrm{Spec}(k)$. It is expected that if $X$ is rational, then all subcategories $\A_i$ should be equivalent to $D^b(Z_i)$, i.e. the Azumaya algebras $B_i$ that appear will be trivial. It is true, for example, when $X$ is a smooth del Pezzo surface over a field of degree at least $5$, see \cite{aberdp}. Because a quintic del Pezzo surface with rational Gorenstein singularities over a field is always rational (see \cite{sbratqdp} for the smooth case and \cite[Theorem 9.1(b)]{ctssdp} for the singular case), one could anticipate that no nontrivial Azumaya algebras would appear in the semiorthogonal decomposition of the quintic del Pezzo fibration and the main result of the paper confirms the anticipation.
\begin{thm}\label{main}
Let $f:\mathcal{X}\to S$  be a flat morphism where each fiber of $f$ is a quintic del Pezzo surface with rational Gorenstein singularities. Then there is an $S$-linear semiorthogonal decomposition 
\begin{equation}\label{sodmain}
D^b(\mathcal{X})=\langle D^b(S), D^b(S), D^b(\mathcal{Z})\rangle
\end{equation}
where $g:\mathcal{Z}\to S$ is flat and finite of degree $5$.

The projection functors of the decomposition have finite cohomological amplitudes and the decomposition is compatible with base change in the sense that for any morphism $h: T\to S$, the quintic del Pezzo fibration $f': \mathcal{X}_T=\mathcal{X}\times_S T\to T$ has an $T$-linear semiorthogonal decomposition
\[D^b(\mathcal{X}_T)=\langle D^b(T),D^b(T), D^b(\mathcal{Z}\times_S T)\rangle.\]
In particular, if $T$ is a geometric point of $S$, the components of the decomposition can be described explicitly by Theorem \ref{ssod}.
\end{thm}

When $\mathcal{X}$ is a quintic del Pezzo surface $X$ with rational Gorenstein singularities over an algebraically closed field, the semiorthogonal decomposition of $D^b(X)$ is obtained by applying \cite{kkssurf}. More concretely, we consider the minimal resolution $\wt{X}$ of $X$ and use a semiorthogonal decomposition of $D^b(\wt{X})$ that is compatible with the contraction $\pi: \wt{X}\to X$. In such a way, the decomposition of $D^b(X)$ can be derived from that of $D^b(\wt{X})$ via $\pi_*$. It turns out that the description only depends on the singular type of $X$. Section \ref{dcqdp} presents the process and the result is given in Theorem \ref{ssod}. 

Moreover, the embedding functors of the components of the above decomposition are given by Fourier-Mukai functors with kernels $\O_X$, a rank $2$ vector bundle $F$ and a rank $5$ vector bundle $\Q$ on $X$ respectively. In section \ref{moduli}, we give a moduli space interpretation for this decomposition. Namely, the interesting components are equivalent to the derived categories of the fine moduli spaces $\M_d(X)$ of semistable sheaves with given Hilbert polynomials $h_d(t), d\in\{2,3\}$ (see (\ref{hilb}) for definitions) and the kernels of the embedding functors are isomorphic to the respective universal families (Theorem \ref{ms}).

To produce a semiorthogonal decomposition for the fibration $f:\mathcal{X}\to S$, we consider the relative moduli spaces $\M_d(\mathcal{X}/S)$ of semistable sheaves with the same Hilbert polynomials $h_d(t)$ and show that they are also fine moduli spaces (Proposition \ref{msf}). Comparing with the case of a single quintic del Pezzo surface, we deduce that the derived categories of these fine relative moduli spaces give the components of (\ref{sodmain}) and the universal families are the kernels of the embedding functors. It is explained in section \ref{msapproach}.

Essentially, to prove that the relative moduli spaces $\M_d(\mathcal{X}/S)$ are fine, one needs to show that the moduli spaces $\M_d(X)$ of a quintic del Pezzo surface $X$ over an arbitrary field $k$ are fine, namely the vector bundles $F, \Q$ on $X\times_k k_s$ (base changed to the separable closure $k_s$) descend to $X$. The case for $\Q$ follows easily from an arithmetic reason (values of the Hilbert polynomial) and the case for $F$ requires the geometry of $X$ (e.g., $X$ has a rational point). Section \ref{global generation},\ref{galois descent} study the properties of $F$ and deduce that it is globally generated and descends to $X$.

Alternatively, we produce a semiorthogonal decomposition of the fibration using the theory of Homological Projective Duality (HPD). In section \ref{hpdapproach}, we provide two constructions for the fibration $f:\mathcal{X}\to S$ to be a linear section of a Grassmannian bundle. The first construction uses the universal family $\mathcal{E}_2$ of the relative moduli space $\M_2(\mathcal{X}/S)$ which induces a morphism from $\mathcal{X}$ to the Grassmannian bundle over $S$. With this construction, HPD produces a semiorthogonal decomposition of $D^b(\mathcal{X})$ which is the same as the one obtained from the moduli space approach, see Theorem \ref{main2}. The second construction uses a vector bundle related to the normal bundle of the anticanonical embedding. In section \ref{coincide}, we prove that these two constructions are isomorphic in characteristic $0$. Consequently, we obtain a relation between the universal family $\mathcal{E}_2$ and the normal bundle of the anticanonical embedding, see Theorem \ref{identical}.

Finally, in order to apply HPD, we need a Lefschetz type semiorthogonal decomposition for $D^b(\Gr(2,5))$. It was only known in characteristic $0$. In the appendix, we verify that it still holds in large characteristic (Proposition \ref{lefschetzgr}) and it is achieved by performing mutations to the Kapranov's collection.

For the convenience of the reader, in section \ref{preliminaries}, we include basic facts about quintic del Pezzo surfaces with rational Gorenstein singularities as well as notions and results of derived categories related to the base change of semiorthogonal decompositions.
%-------Notations------------------------
\subsection*{Notations}
Denote by $D(Y), D^-(Y), D^+(Y),D^b(Y)$ the unbounded, bounded above, bounded below and bounded derived categories of quasi-coherent sheaves on the scheme $Y$ with coherent cohomology.  Given $G\in D(Y)$, denote the $p$-th sheaf cohomology of $G$ by $\mathcal{H}^p(G)$. For $p,q\in\mathbb{Z},p\leqslant q$, write $D^{[p,q]}(Y)=\{G\in D(Y)\,|\, \mathcal{H}^i(G)=0, i\notin [p,q] \}$. For a triangulated subcategory $T$ of $D^b(Y)$, denote its right orthogonal (resp. its left orthogonal) by $T^{\perp} = \{G\in D^b(Y)\,|\,\RHom(A,G) = 0, \forall A\in T\}$ (resp. ${}^\perp T=\{G\in D^b(Y)\,|\, \RHom(G,A) = 0, \forall A\in T\}$). 

For a morphism $f: Y\to W$, denote by $f^*, f_*$ the derived pull-back and push-forward. The usual pull-back and push-forward of morphisms will be denoted by $L_0f^*, R^0f_*$.
\subsection*{Acknowdgements}
I would like to thank Marcello Bernardara and Alexander Kuznetsov for helpful conversations. Especially I thank Kuznetsov for the idea of the second construction in \S\ref{construction2} and for pointing out a mistake in a previous draft. I am thankful for the encouragement and support from Charles Vial. Finally, I thank the referee for helpful comments. I was partially supported by the Bielefeld Young Researchers' Fund for this work. 
%--------Preliminaries-------------------------
\section{Preliminaries}\label{preliminaries}
%--------Quintic Del Pezzo Surfaces-------
\subsection{Quintic Del Pezzo Surfaces}\label{qdp}
Assume $k$ is an algebraically closed field. 

We recall some basic properties of quintic del Pezzo surfaces. For more details, see \cite[\S 8.5]{dolcag}\cite[\S 25-26]{cubicforms}. Let $X$ be a quintic del Pezzo surface over $k$ with rational Gorenstein singularities (over algebraically closed fields, rational Gorenstein singularity is equivalent to du Val, ADE singularity or rational double point). Let $\pi:\wt{X}\to X$ be its minimal resolution. Then $\pi$ is crepant ($\pi^*K_X=K_{\wt{X}}$), $\wt{X}$ is a weak del Pezzo surface ($-K_{\wt{X}}$ is nef and big) and we have
\[\wt{X}=X_5\to X_4\to\dots\to X_1=\p^2\]
where $X_{i+1}\to X_i$ is the blow-up of $X_i$ at the point $x_i$. Let $h$ be the hyperplane class on $\mathbb{P}^2$ as well as its pull-back to $X_i, i\geqslant 2$. Denote by $e_i,1\leqslant i \leqslant 4$ the classes of pull-backs of exceptional divisors $E_i$ over $x_i$ to $X_j, j\geqslant i$. Then $\mathrm{Pic}(\wt{X})=\mathbb{Z} h\oplus\bigoplus_{i=1}^4 \mathbb{Z} e_i$. The canonical divisor $K_{\wt{X}}=-3h+\sum_{i=1}^4 e_i$ and $h^2=1, e_i^2=-1, h.e_i=0, e_i.e_j=0$ for $i\neq j$. The orthogonal complement $R=K_{\wt{X}}^{\perp}\subset\mathrm{Pic}{\wt{X}}\otimes_{\mathbb{Z}}\mathbb{R}$ equipped with the scalar product (intersection product but with the opposite sign) is the root system $A_4$. The simple roots are $e_1-e_2, e_2-e_3, e_3-e_4, h-e_1-e_2-e_3$ and the Weyl group is the permutation group $S_5$.

The possible configurations of the points $x_i$ are (the notation $x>y$ represents that $x$ is an infinitely near point over $y$):

(I) $x_1,x_2,x_3,x_4$ are proper points of $\p^2$;

(II) $x_2>x_1, x_3,x_4$;

(III) $x_2>x_1, x_4>x_3$;

(IV) $x_3>x_2>x_1, x_4$;

(V) $x_4>x_3>x_2>x_1$.

Write $\Delta_{ijl}\in |h-e_i-e_j-e_l|$ and $\Delta_{ij}\in |e_i-e_j|$ for members in respective classes. The
singular types of $X$ and the corresponding sets $\Delta$ of $(-2)$-curves (effective $(-2)$-classes) on $\wt{X}$ are:

(I.1) smooth; (I.2) $A_1$ with $\Delta_{123}$;

(II.1) $A_1$ with $\Delta_{12}$; (II.2) $A_1+A_1$ with $\Delta_{12},\Delta_{123}$; (II.3) $A_2$ with $\Delta_{12},\Delta_{134}$;

(III.1) $A_1+A_1$ with $\Delta_{12}, \Delta_{34}$;  (III.2) $A_1+A_2$ with $\Delta_{12}, \Delta_{34},
\Delta_{123}$;

(IV.1) $A_2$ with $\Delta_{12}, \Delta_{23}$;  (IV.2) $A_1+A_2$ with $\Delta_{123}, \Delta_{12},
\Delta_{23}$;  (IV.3) $A_3$ with $\Delta_{12}, \Delta_{23}, \Delta_{124}$;

(V.1) $A_3$ with $\Delta_{12}, \Delta_{23}, \Delta_{34}$;  (V.2) $A_4$ with $\Delta_{12}, \Delta_{23}, \Delta_{34},\Delta_{123}$.
%-----------Derived Categories--------------
\subsection{Derived Categories}
For the homological background, one can refer to \cite[\S2]{kuzhs}. All schemes are noetherian. 

We recall the notions of Tor-/Ext-/cohomological amplitudes. Let $f: Y\to W$ be a proper morphism of separated schemes. The right adjoint $f^!$ of $f_*$ exists \cite{ngd}. An object $G\in D(Y)$ has \textit{finite} Tor-\textit{amplitude} over $W$ if there exist $p,q\in\mathbb{Z}$ such that $G\otimes_{\O_Y} f^*H\in D^{[a+p,b+q]}(Y)$ for any $H\in D^{[a,b]}(W)$. An object $G\in D(Y)$ has \textit{finite} Ext-\textit{amplitude} over $W$ if there exist $p,q\in\mathbb{Z}$ such that $R\mathcal{H}om_{\O_Y}(G, f^!H)\in D^{[a+p,b+q]}(Y)$ for any $H\in D^{[a,b]}(W)$. Given schemes $Y, W$, let $T\subset D(Y)$ be a triangulated subcategory. A triangulated functor $\Phi: T\to D(W)$ has \textit{finite cohomological amplitude} if for any $a,b\in\mathbb{Z}$, there exist $p,q\in\mathbb{Z}$ not depending on $a,b$ such that $\Phi(T\cap D^{[a,b]}(Y))\subset D^{[a+p,b+q]}(W)$.

Let $Y, W$ be proper schemes. Let $K\in D^-(Y\times W)$ and $p: Y\times W\to Y, q: Y\times W\to W$ be projections. Define functors $\Phi_K=q_*(p^*-\otimes_{\O_{Y\times W}}K): D^-(Y)\to D^-(W)$ and $\Phi_K^!=p_*R\mathcal{H}om_{\O_{Y\times W}}(K, q^!-): D^+(W)\to D^+(Y)$.
\begin{lemma}{\cite[Lemma 2.4]{kuzhs}\cite[Lemma 2.10]{kuzbc}}
If $K$ has finite Tor-amplitude over $Y$ and finite Ext-amplitude over $W$, then 

(i) $\Phi_K$ takes $D^b(Y)$ to $D^b(W)$, $\Phi_K^!$ takes $D^b(W)$ to $D^b(Y)$ and $\Phi^!$ is the right adjoint of $\Phi_K$;

(ii) $\Phi_K, \Phi_K^!$ have finite cohomological amplitudes.
\end{lemma}
\begin{lemma}
Assume $D^b(Y)=\langle T_1,\dots, T_n\rangle$ is a semiorthogonal decomposition of a scheme $Y$ by right admissible subcategories $T_i$, that is, the embedding functors $\beta_i: T_i\hookrightarrow D^b(Y)$ have right adjoint functors, which we denote by $\beta_i^!: D^b(Y)\to T_i$. If for $i\geqslant 2, \beta_i^!$ have finite cohomological amplitudes, then so are projection functors $\gamma_j: D^b(Y)\to D^b(Y)$ to the $j$-th component $T_j$ for all $j$. 
\end{lemma}
\begin{proof}
Denote the embedding $\langle T_1,\dots,T_i\rangle\hookrightarrow D^b(Y)$ by $\alpha_i$ and its left adjoint by $\alpha_i^*$. Then $\gamma_i=\beta_i\circ\beta_i^!|_{\langle T_1,\dots, T_i\rangle}\circ\alpha_i^*=\beta_i\circ\beta_i^!\circ\alpha_i\circ\alpha_i^*$. In particular, $\gamma_n=\beta_n\circ\beta_n^!$ and $\gamma_1=\alpha_1\circ\alpha_1^*$. The semiorthogonal decomposition provides the exact triangles for each $G\in D^b(Y)$:
\[\gamma_n G\to G\to \alpha_{n-1}\alpha_{n-1}^*G,\]
\[\gamma_{n-1}G\to\alpha_{n-1}\alpha_{n-1}^*G\to \alpha_{n-2}\alpha_{n-2}^*G,\text{etc.}\]
Recursively, we deduce $\alpha_i\circ\alpha_{i}^*$ and thus $\gamma_i$ have finite cohomological amplitudes.
\end{proof}
We give a special version of the base change of semiorthogonal decompositions, which follows from Theorem 5.6,6.4 in \cite{kuzbc}.
\begin{prop}\label{bc}
Let $f: Y\to S, f_1: Y_1\to S,\dots, f_n: Y_n\to S$ be flat projective morphisms. Assume $K_i\in D^-(Y\times_S Y_i)$ have finite Tor-amplitudes over $Y$ and finite Ext-amplitudes over $Y_i$ for all $i$. Let $g: T\to S$ be any base change. Denote base change along $g$ by subscript $T$, i.e. $K_{iT}\in D^-(Y_T \times_T Y_{iT})$. If there is an $S$-linear semiorthogonal decomposition 
\begin{equation}\label{sody}
D^b(Y)=\langle\Phi_{K_1}(D^b(Y_1)),\dots,\Phi_{K_n}(D^b(Y_n))\rangle,
\end{equation}
then the projection functors of (\ref{sody}) have finite cohomological amplitudes and there is a $T$-linear semiorthogonal decomposition
\[D^b(Y_T)=\langle\Phi_{K_{1T}}(D^b(Y_{1T})),\dots,\Phi_{K_{nT}}(D^b(Y_{nT}))\rangle.\]
\end{prop}
\begin{proof}
We only need to check the projection functors of (\ref{sody}) have finite cohomological amplitudes. This follows from $\Phi_{K_i}^!$ having finite cohomological amplitudes by previous lemmas.
\end{proof}
%---------Derived category of a quintic del Pezzo surface------------
\section{Derived Category of a Quintic Del Pezzo Surface}\label{dcqdp}
Assume $k$ is an algebraically closed field. We adopt the same notation as \S\ref{qdp}.

To find a semiorthogonal decomposition of $D^b(X)$, we will apply the method in \cite{kkssurf}, which is a generalization of \cite[\S3]{kuzdp6}. In order to utilize the method, we will first verify that the decomposition of $D^b(\wt{X})$ given in \cite[Proposition 4.2]{karnog} when $X$ is smooth (thus $\wt{X}=X$) also works for any weak del Pezzo surface and then check that the decomposition is \textit{compatible with the contraction} $\pi$.

Since $X_2$ is the blow-up of $\p^2$ at a point, we have a semiorthogonal decomposition
\begin{equation*}
\begin{array}{rl}
D^b(X_2) &= \langle \O_{e_1}(-1), \O_{X_2}(-h), \O_{X_2}, \O_{X_2}(h)\rangle\\
		  &= \langle\O_{X_2}(-h), \O_{X_2}(e_1-h), \O_{X_2}, \O_{X_2}(h)\rangle\\
\end{array}
\end{equation*}
where the right mutation of the pair $(\O_{e_1}(-1), \O_{X_2}(-h))$ produces the second equality. Repeating the process, we obtain
\begin{equation*}
\begin{array}{l}
D^b(\wt{X})\\
= \langle\O_{\wt{X}}(-h), \O_{\wt{X}}(e_4-h), \O_{\wt{X}}(e_3-h), \O_{\wt{X}}(e_2-h), \O_{\wt{X}}(e_1-h), \O_{\wt{X}}, \O_{\wt{X}}(h)\rangle\\
= \langle\O_{\wt{X}}(e_4-h), \O_{\wt{X}}(e_3-h), \O_{\wt{X}}(e_2-h), \O_{\wt{X}}(e_1-h), \O_{\wt{X}}, \O_{\wt{X}}(h), \O(-K_{\wt{X}}-h)\rangle 
\end{array}
\end{equation*}
where the second equality is achieved by mutating $\O_{\wt{X}}(-h)$ from the leftmost to the rightmost position and the effect of the mutation is tensoring by $\O_{\wt{X}}(-K_{\wt{X}})[-2]$. Next, in the order of $i=4$ to $1$, mutate $\O_{\wt{X}}(e_i-h)$ to the rightmost position. Finally, mutate $\O_{\wt{X}}(-K_{\wt{X}}-h)$ to the left side of $\O_{\wt{X}}(h)$. Therefore, we obtain the following semiorthogonal decomposition
\begin{equation}\label{sodwdp}
D^b(\wt{X})= \langle\wt{\A}_1, \wt{\A}_2,\wt{\A}_3\rangle.
\end{equation}
Here $\wt{\A}_1= \langle\O_{\wt{X}}\rangle, \wt{\A}_2=\langle\wt{F}\rangle$ where $\wt{F}$ is
the unique nontrivial extension of 
\begin{equation}\label{uex}
0\to\O_{\wt{X}}(-K_{\wt{X}}-h)\to \wt{F}\to\O_{\wt{X}}(h)\to 0
\end{equation}
and
\begin{equation}\label{a3}
\wt{\A}_3=\langle\O_{\wt{X}}(h),\O(e_4-K_{\wt{X}}-h),\O(e_3-K_{\wt{X}}-h),\O(e_2-K_{\wt{X}}-h),\O(e_1-K_{\wt{X}}-h)\rangle.
\end{equation}
Furthermore, the push-forward of the resolution map $\pi_*: D^b(\wt{X})\to D^b(X)$ is essentially surjective with
$\ker(\pi_*)=\langle\O_{\Delta}(-1)\rangle^\oplus$ where $\Delta$ ranges through the set of
$(-2)$-curves and $\langle\rangle^{\oplus}$ denotes the minimal triangulated subcategory closed under infinite direct sums \cite[Lemma 2.3 and Corollary 2.5]{kuzdp6}. 
\begin{lemma}{\cite[Lemma 2.5]{kkssurf}}\label{pf}
Let $Y$ be a normal surface with rational singularities and let $p: \wt{Y}\to Y$ be its resolution. Let $\wt{G}\in D^b(\wt{Y})$. Then the following properties are equivalent:

(1) for any irreducible exceptional divisor $E$ of $p$ one has $\wt{G}|_E\in \langle\O_E\rangle$;

(2) for any irreducible exceptional divisor $E$ of $p$ one has $\Ext^*(\wt{G}|_E, \O_E(-1))=0$;

(3) there exists $G\in D^{perf}(Y)$ a perfect complex such that $\wt{G}\cong p^*G$;

(4) one has $p_*\wt{G}\in D^{perf}(Y)$ and $\wt{G}\cong p^*(p_*\wt{G})$.

In addition, if $\wt{G}$ is a pure sheaf or a locally free sheaf, then so is $p_*\wt{G}$.
\end{lemma}
From its construction by the exact sequence (\ref{uex}),  one checks that the locally free sheaf $\wt{F}$ on $\wt{X}$ satisfies Lemma \ref{pf}(1). Therefore,
\begin{equation}\label{vb}
F: =\pi_*\wt{F}
\end{equation} 
is a locally free sheaf of rank $2$ on $X$ and $\wt{F}=\pi^*F$. 
\begin{defn}{\cite[Definition 2.7]{kkssurf}}
Let $Y$ be a normal surface with rational singularities and let $p: \wt{Y}\to Y$ be its resolution. A semiorthogonal decomposition $D^b(\wt{Y})=\langle T_1, \dots,T_n\rangle$ is \textit{compatible with the contraction} $p$ if for each irreducible component $E$ of the exceptional locus one has
\[\O_E(-1)\in T_i\]
for one of the the components $T_i$ of the decomposition.
\end{defn}
For $1\leqslant a\leqslant 3$ and $\{i,j,k,l\}=\{1,2,3,4\}$, there are short exact sequences:
\[0\to \O_{\wt{X}}(e_{a+1}-K_{\wt{X}}-h)\to \O_{\wt{X}}(e_a-K_{\wt{X}}-h)\to \O_{\Delta_{a, a+1}}(-1)\to 0,\]
\[0\to \O_{\wt{X}}(h)\to\O_{\wt{X}}(e_l-K_{\wt{X}}-h)\to \O_{\Delta_{ijk}}(-1)\to 0.\]
Therefore, the semiorthogonal decomposition (\ref{sodwdp}) is compatible with the contraction $\pi$. Theorem 2.12 in \cite{kkssurf} indicates that
\begin{equation}\label{soddp}
D^b(X)=\langle\A_1, \A_2, \A_3\rangle
\end{equation}
where $\A_i=\pi_*(\wt{\A}_i)$ are admissible subcategories. In fact, the components $\wt{A}_i$ can be described explicitly. Note that both $\A_1=\langle\O_X\rangle$ and $\A_2=\langle F\rangle$ are generated by exceptional objects ($F$ is exceptional because $\wt{F}$ is by computation using sequence (\ref{uex})) and thus equivalent to $D^b(\mathrm{Spec}(k))$.

To describe $\A_3$, we observe that $\wt{\A}_3$ has an \textbf{orthogonal} decomposition of the form 
\begin{equation}
\wt{\A}_3=\langle\wt{\B}_1, \dots, \wt{\B}_n\rangle
\end{equation}
such that each $\wt{\B}_{q}=\langle\L_0, \L_1, \dots, \L_m\rangle$ is generated by line bundles $\L_p$ with the relation $\L_p=\L_0(E_1+\dots +E_p)$. Here $\{E_1, \dots, E_m\}$ is a chain of $(-2)$-cuves on $\wt{X}$. Moreover, they fit into short exact sequences
\begin{equation}\label{lb}
0\to\L_{p-1}\to\L_p\to \O_{E_p}(-1)\to 0.
\end{equation}
Such $\wt{\B_q}$ is said to be \textit{untwisted adherent} to the chain $\cup_{i=1}^m E_i$ in \cite[Definition 3.6]{kkssurf}.

For singular types (*.1), (*.2), the components $\wt{\B}_q$ are obtained by regrouping line bundles in the decomposition (\ref{a3}). For example, for (II.2), we have $\wt{\A}_3=\langle\wt{\B}_1, \wt{\B}_2, \wt{\B}_3\rangle$ where $\wt{\B}_1=\langle\O_{\wt{X}}(h), \O_{\wt{X}}(e_4-K_{\wt{X}}-h)\rangle, \wt{\B}_2=\langle\O_{\wt{X}}(e_3-K_{\wt{X}}-h)\rangle, \wt{\B}_3=\langle\O_{\wt{X}}(e_2-K_{\wt{X}}-h), \O_{\wt{X}}(e_1-K_{\wt{X}}-h)\rangle$.

For type (II.2), because $\O_{\wt{X}}(h)$ is orthogonal to $\O_{\wt{X}}(e_i-K_{\wt{X}}-h)$ for $i=3,4$, the right mutations do not alter these line bundles. Hence,
\[\wt{\A}_3 =\langle\O(e_4-K_{\wt{X}}-h)\,|\,\O(e_3-K_{\wt{X}}-h)\,|\,\O_{\wt{X}}(h),\O(e_2-K_{\wt{X}}-h),\O(e_1-K_{\wt{X}}-h)\rangle\]
where it is divided into $3$ subgroups separated by $|$. Similarly, for type (IV.3), we have
\[\wt{\A}_3 =\langle\O(e_4-K_{\wt{X}}-h)\, | \,\O_{\wt{X}}(h),\O(e_3-K_{\wt{X}}-h),\O(e_2-K_{\wt{X}}-h),\O(e_1-K_{\wt{X}}-h)\rangle.\]

Below we give a brief summary of the procedure in \cite[\S 3]{kkssurf} for obtaining the explicit description of $\A_3$.

The line bundles $\{\L_0, \dots, \L_m\}$ satisfy
\begin{equation*}
\Ext^{\bullet}(\L_i, \L_j)=
\left\{
\begin{array}{ll}
k\oplus k[-1], & j\geqslant i+1\\
k, & j=i\\
0, & j< i
\end{array}
\right.
\end{equation*}
Define $\P_0$ as the \textit{iterated extension} of the collection $\{\L_0,\dots, \L_m\}$ as follows. Set $\P_m=\L_m$ and $\P_{m-1}$ to be the unique nontrivial extension of $\P_m$ by $\L_{m-1}$, i.e.
\[0\to\P_m\to\P_{m-1}\to\L_{m-1}\to 0.\] 
Notice that inductively one has $\Ext^{\bullet}(\L_i,\P_j )=k\oplus k[-1]$ for $j\geqslant i+1$. Finally, $\P_0$ is the unique nontrivial extension
\[0\to\P_1\to\P_0\to\L_0\to0.\]
For later use, we observe the following property for the vector bundle $\P_0$:
\begin{lemma}\label{vbex}
$\Q_0:=\pi_*\P_0$ is a vector bundle and $\P_0=\pi^*\Q_0$.
\end{lemma}
\begin{proof}
It suffices to check that $\P_0$ satisfies condition (1) of Lemma \ref{pf}, i.e. $\P_0|_{E_p}\in\langle\O_{E_p}\rangle$ for $1\leqslant p\leqslant m$. Recall that $\langle\L_0,\dots,\L_m\rangle$ is untwisted adherent to the chain of $(-2)$-curves $\cup_{i=1}^m E_i$. Then \cite[Lemma 3.5(1)]{kkssurf} indicates that $\L_0\cdot E_1=1$ and $\L_0\cdot E_p=0$ for $2\leqslant p\leqslant m$. Hence,
\[
\L_p\cdot E_q=
\left\{
\begin{array}{ll}
1, & p=q-1\\
-1, & p=q\\
0, & \text{otherwise}
\end{array}
\right.
\]
and one can show inductively that $\P_p|_{E_q}\in \langle\O_{E_q}\rangle$ for $q\geqslant p+1$.
\end{proof}
Furthermore, the direct sum $T=\bigoplus_{j=0}^m \P_j$ is the \textit{universal extension} for the collection $\{\L_0,\dots, \L_m\}$ and is a tilting bundle for the subcategory $\wt{\B}_q$ \cite[Theorem 2.5]{hpnc}. Therefore, there is an induced equivalence:
 \begin{equation}
 \tilde{\beta}=\mathrm{RHom}(T, -): \wt{\B}_q\overset{\simeq}{\longrightarrow} D^b(\textrm{mod-}\Lambda)
 \end{equation}
where $\Lambda=\mathrm{End}(T)$ and $D^b(\textrm{mod-}\Lambda)$ is the derived category of finite right modules over $\Lambda$. In this case, $\Lambda$ is the Auslander algebra of $k[x]/x^{m+1}$. Note that $\wt{\beta}$ induces the equivalence $\wt{\B_q}^-\simeq D^-(\textrm{mod-}\Lambda)$ as well.

Let $P_0=\tilde{\beta}(\P_0)$ and define $K=\mathrm{End}_{\tilde{X}}(\P_0)=\mathrm{End}_{\Lambda}(P_0)$. Then $P_0$ is a $K$-$\Lambda$-bimodule and we have functors:
\[
\begin{array}{ll}
\rho_*: D^-(\textrm{mod-}\Lambda) \to D^-(\textrm{mod-}K), & M\mapsto \mathrm{RHom}_{\Lambda}(P_0, M)\\
\rho^*: D^-(\textrm{mod-}K)\to D^-(\textrm{mod-}\Lambda), & N\mapsto N\otimes_K P_0
\end{array}
\]
Denote $\pi_*(\wt{\B}_q)$ by $\B_q$. Since the orthogonal decomposition $\wt{\A_3}=\langle\wt{\B_1}, \dots, \wt{\B_n}\rangle$ is compatible with the contraction $\pi$, we obtain an orthogonal decomposition
\begin{equation}\label{od}
\A_3=\langle\B_1,\dots, \B_n\rangle.
\end{equation}
Let $\wt{\alpha}=-\otimes_{\Lambda} T$ be the inverse of $\wt{\beta}$. Theorem 3.16 in \cite{kkssurf} proves that the functor 
\[\alpha=\pi_*\circ\wt{\alpha}\circ\rho^*: D^-(\textrm{mod-}K)\to\B_q^-\]
induces an equivalence
\begin{equation}
\alpha=\pi_*\circ\wt{\alpha}\circ\rho^*: D^b(\textrm{mod-}K)\overset{\simeq}{\longrightarrow}\B_q
\end{equation}
where $K=k[x]/x^{m+1}$. Since $K$ is a compact generator of the category $D^b(\textrm{mod-}K)$, the image $\alpha(K)=\pi_*\circ\wt{\alpha}(P_0)=\pi_*(\P_0)=\Q_0$ is a compact generator of $\B_q$. From the construction, one sees that $\alpha=-\otimes_{K}\Q_0$ is a Fourier-Mukai functor with kernel $\Q_0$. More generally, one has $\alpha(k[x]/x^{p+1})=\pi_*(\P_{m-p})$ for $0\leqslant p\leqslant m$. In particular, $\alpha(k)=\pi_*(\L_m)$. Since $\O_{E_p}(-1)\in\ker\pi_*$, sequences (\ref{lb}) imply that $\pi_*(\L_0)=\dots=\pi_*(\L_m)$ and thus $\B_q=\langle\pi_*(\L_0)\rangle$. Geometrically speaking, if we identify $D^b(\textrm{mod-}K)$ with $D^b(\mathrm{Spec}(K))$, then $\pi_*(\L_0)$ is the image of the unique closed point of $\mathrm{Spec}(K)$.

Now rename $\P_0, \Q_0$ by $\P_0^q, \Q_0^q$ to indicate that they are constructed from $\wt{\B_q}$. Recall that we have the orthogonal decomposition $\wt{\A_3}=\langle\wt{\B_1}, \dots, \wt{\B_n}\rangle$. Define
\begin{equation}\label{tvbex}
\P=\oplus_{q=1}^n \P_0^q, \quad \Q= \oplus_{q=1}^n \Q_0^q.
\end{equation}
\begin{lemma}\label{cg}
$\Q=\pi_*\P$ is a vector bundle such that $\P=\pi^*\Q$ and a compact generator of $\A_3$. In addition, $\A_2= {}^{\perp}\O_X\cap \Q^{\perp}$ and $\A_3={}^{\perp}\O_X\cap {}^{\perp}F$.
\end{lemma}
\begin{proof}
It follows from the argument above, Lemma \ref{vbex} and decompositions (\ref{soddp}), (\ref{od}).
\end{proof}
To summarize the discussion of the section, we have
\begin{thm}\label{ssod}
Let $X$ be a quintic del Pezzo surface with rational Gorenstein singularities over an algebraically closed field $k$. Then the derived category $D^b(X)$ only depends on the singular type of $X$ and it has the following semiorthogonal decomposition:
\[D^b(X)=\langle D^b(\mathrm{Spec}(k)), D^b(\mathrm{Spec}(k)), D^b(Z)\rangle\]
where $Z=\bigsqcup\mathrm{Spec}(\frac{k[x]}{x^{p+1}})$ is an affine scheme of length $5$. A singular point of type $A_p$ on $X$ contributes a singular point $\mathrm{Spec}(\frac{k[x]}{x^{p+1}})$ of length $p+1$ on $Z$. 

More explicitly, if $X$ is smooth, then $Z=\mathrm{Spec}(k^5)$; if $X$ has singular type $A_p$, then $Z=\mathrm{Spec}(k^{4-p}\times\frac{k[x]}{x^{p+1}})$; if $X$ has singular type $A_p+A_q$, then $Z=\mathrm{Spec}(k^{3-p-q}\times\frac{k[x]}{x^{p+1}}\times\frac{k[x]}{x^{q+1}})$.

Moreover, the embeddings of components are given by Fourier-Mukai functors with kernels $\O_X$, $F$ defined by (\ref{vb}) and $\Q$ defined by (\ref{tvbex}) respectively.
\end{thm}
%--------Moduli Space Interpretation------------
\section{Moduli Space Interpretation}\label{moduli}
We use the same notation as section \ref{dcqdp}. For a sheaf $\mathcal{F}$ on $X$, denote $h_{\mathcal{F}}(t)$ the Hilbert polynomial of $\mathcal{F}$ with respect to the ample divisor $-K_X$. That is,
\begin{equation*}
h_{\mathcal{F}}(t)=\chi(\mathcal{F}(-tK_X))\in\mathbb{Z}[t].
\end{equation*}
More generally, for a bounded complex of sheaves $\mathcal{F}^{\bullet}$ on $X$, one has
\[h_{\mathcal{F}^{\bullet}}(t)=\sum (-1)^i \chi(\mathcal{F}^i(-tK_X))=\sum (-1)^i\chi(\mathcal{H}^i(\mathcal{F}^{\bullet})(-tK_X))\]
where $\mathcal{H}^i(\mathcal{F}^{\bullet})$ is the $i$-th cohomology sheaf. Let $\G$ be a sheaf on $\wt{X}$. The Leray spectral sequence $H^i(X, R^j\pi_*\G)\Rightarrow H^{i+j}(\wt{X},\G)$ implies that $h_{\pi_*\G}(t)=\chi(\G(-tK_{\wt{X}}))$.

By Riemann-Roch, given a Cartier divisor $D$ on $\wt{X}$, the Hilbert polynomial for $\pi_*\O(D)$ is 
\begin{equation*}
\begin{array}{rl}
h_{\pi_*\O(D)}(t) & = \frac{K_X^2}{2}t(t+1)-\frac{K_X\cdot D}{2}(2t+1)+\frac{D^2}{2}+\chi(\O_X)\\ 
		   & = \frac{5}{2}t^2+(\frac{5}{2}-K_X\cdot D)t+\frac{D^2-K_X\cdot D}{2}+1\\
\end{array}
\end{equation*}

By calculation, the generators of $\A_3$ have the same Hilbert polynomial. Recall that $\A_2=\langle F\rangle$, we denote Hilbert polynomials of the generators of $\A_i, i=2,3$ by
\begin{equation}\label{hilb}
\left\{
\begin{array}{l}
h_2(t) : = h_F(t) = 5(t+1)^2\\
h_3(t) : = h_{\pi_*\O(h)}(t) = \frac{1}{2}(t+1)(5t+6)\\
\end{array}
\right.
\end{equation}
\begin{lemma}
Let $\pi: \wt{X}\to X$ be the minimal resolution. Then $\pi_*\O(D)$ for $D=h, e_i-K_{\wt{X}}-h, 1\leqslant i \leqslant 4$ are stable sheaves of rank $1$ with Hilbert polynomial $h_3(t)$ and $F$ is a stable bundle of rank $2$ with Hilbert polynomial $h_2(t)$. 
\end{lemma}
\begin{proof}
Lemma \ref{r1van} suggests that $\pi_*\O(D)=R^0\pi_*\O(D)$ are sheaves. They are stable because they are torsion free rank $1$ sheaves. For the stability of $F$, we will use the equivalent criterion in \cite[Proposition 1.2.6]{hlmoduli} and check conditions for all proper torsion free quotient sheaves of $F$.

Let $G$ be a torsion free proper quotient sheaf of $F$ and denote $\wt{G}=L_0\pi^*G$. Factor $\wt{G}$ by torsion subsheaf $T$ and torsion free quotient sheaf $G'$:
\[0\to T\to \wt{G}\to G'\to 0.\]
Since $G=R^0\pi_*\wt{G}$ is torsion free and $R^0\pi_*T$ is torsion, the sheaf $R^0\pi_*T=0$. Now consider the following commutative diagram with exact rows (taking $\wt{E}, E'$ as corresponding kernels):
\begin{center}
\begin{tikzcd}
0 \arrow{r} & \wt{E} \arrow{r}\arrow{d} & \wt{F} \arrow{r}\arrow[equal]{d} & \wt{G} \arrow{r}\arrow{d} & 0\\
0 \arrow{r} & E' \arrow{r} & \wt{F} \arrow{r} & G' \arrow{r} & 0.\\
\end{tikzcd}
\end{center}
Pushing-forward the first row along $\pi$ induces the exact sequence
\[0\to R^0\pi_*\wt{E}\to F\to G\to R^1\pi_*\wt{E}\to 0.\]
Since $F\to G$ is surjective, one gets $R^1\pi_*\wt{E}=0$. In addition, the diagram induces the short exact sequence
\[0\to \wt{E}\to E'\to T\to 0.\]
Therefore, we have $\pi_*\wt{E}\cong R^0\pi_*\wt{E}\cong R^0\pi_*E'$, which is the kernel of the quotient map $F\to G$. It also implies that $R^1\pi_*E'\cong R^1\pi_*T$, which is nonzero unless $T=0$.

As a torsion free proper quotient sheaf, the sheaf $G$ is of rank $1$. Thus sheaves $G', E'$ are also of rank $1$. By \cite[Proposition 1.1 and Corollary 1.4]{hstable}, the sheaf $E'$ is reflexive and thus locally free of rank $1$. Since $h_F(t)=h_G(t)+h_{R^0\pi_*E'}(t)$, the stability condition $\frac{1}{2}h_F(t)<h_G(t)$ holds if and only if $h_{R^0\pi_*E'}(t)<\frac{1}{2}h_F(t)$. Since the support of $R^1\pi_*E'$ is zero-dimensional, the difference $h_{\pi_*E'}-h_{R^0\pi_*E'}=-h_{R^1\pi_*E'}$ is a constant. Therefore, it is enough to show that the coefficient of degree $1$ term of $h_{\pi_*E'}$ is less than that of $\frac{1}{2}h_F$, which is $5$.

Recall that $\wt{F}$ is defined by the extension (\ref{uex}). Then the composition $E'\to \wt{F}\to \O_{\wt{X}}(h)$ is either $0$ or injective. The first case implies that $E'$ is a subsheaf of $\O(-K_{\wt{X}}-h)$ with torsion quotient sheaf. Since the leading coefficient of the Hilbert polynomial of a sheaf is always positive, one has $h_{\pi_*E'}(t)\leqslant h_{\pi_*\O(-K_{\wt{X}}-h)}(t)<\frac{1}{2}h_F(t)$. Otherwise, we have $E'=\O(h-C)$ for some effective divisor $C$. Because $\wt{F}$ is the nontrivial extension, $C\ne 0$. Furthermore, the map $E'\to\wt{F}\to \O_{\wt{X}}(h)\to \O(-K_{\wt{X}}-h)[1]$ being $0$ implies that 
\begin{equation}\label{h1}
\mathrm{Ext}^1(E', \O(-K_{\wt{X}}-h))=H^1(\O(-K_{\wt{X}}-2h+C))=0.
\end{equation}
By calculation, we have 
\[2h_{\pi_*{E'}}(t)=5t^2+(11+2C\cdot K_{\wt{X}})t+6-2h\cdot C+C^2+C\cdot K_{\wt{X}}.\]
Assume the coefficient of degree $1$ term of $h_{\pi_*E'}(t)$ is greater than or equal to $5$. Then $C\cdot K_{\wt{X}}=0$ and it implies that $C$ is a nonnegative $\mathbb{Z}$-linear combination of classes $h-e_1-e_2-e_3$ and $e_i-e_{i+1}, 1\leqslant i \leqslant 3$. It is easy to check that $C^2<0$ and $C\cdot (-K_{\wt{X}}-2h)\leqslant 0$. Hence, $H^0(\O_C(-K_{\wt{X}}-2h+C))=0$. Consider the short exact sequence
\[0\to \O(-K_{\wt{X}}-2h)\to \O(-K_{\wt{X}}-2h+C)\to \O_C(-K_{\wt{X}}-2h+C)\to 0. \]
It induces $0\to H^1(\O(-K_{\wt{X}}-2h))=k\to H^1(\O(-K_{\wt{X}}-2h+C))$, which contradicts (\ref{h1}).
\end{proof}
\begin{lemma}\label{ss}
Let $\G$ be a sheaf on $X$ obtained as the iterated extension of a collection of torsion free semistable sheaves $\{\G_0,\dots, \G_m\}$. Assume that the reduced Hilbert polynomials $h_{\G_i}/\mathrm{rank}(\G_i)$ are equal for all $i$. Then $\G$ is semistable.
\end{lemma}
\begin{proof}
Denote $a=h_{\G_i}/\mathrm{rank}(\G_i)$. It suffices to prove for $m=1$. Then we have the short exact sequence
\[0\to\G_1\to \G\to \G_0\to 0.\]
As the extension, the sheaf $\G$ is also torsion free with reduced Hilbert polynomial $a$. Assume $\G$ is not semistable. Then from Harder-Narasimhan filtration, there exists a semistable subsheaf $\mathcal{F}$ of $\G$ such that $b:=h_{\mathcal{F}}/\mathrm{rank}(\mathcal{F})>a$. Since $b>a$, by semistability, the composition $\mathcal{F}\hookrightarrow \G \to \G_0$ is zero. Thus, $\mathcal{F}$ is a subsheaf of $\G_1$, which contradicts to the assumption that $\G_1$ is semistable.
\end{proof}
\begin{lemma}
Let $\G\in D^b(X)$. Let $x\in X$ be a smooth point. Recall that $F, \Q$ are vector bundles constructed from (\ref{vb})(\ref{tvbex}), we have 
\[\chi(F,\G)=\chi(\O_X, \G)+\chi(\O_X, \G(K_X))-2\chi(\O_x, \G),\]
\[\chi(\Q, \G)=2\chi(\O_X, \G)+3\chi(\O_X, \G(K_X))-5\chi(\O_x, \G).\]
\end{lemma}
\begin{proof}
Since $\pi_*: D^b(\wt{X})\to D^b(X)$ is essentially surjective, there exists $\wt{\G}\in D^b(\wt{X})$ such that $\G=\pi_*\wt{\G}$. The adjunction implies that $\chi(F, \G)=\chi(F, \pi_*\wt{\G})=\chi(\pi^*F, \wt{\G})=\chi(\wt{F}, \wt{\G})$. Similarly, one has $\chi(\Q, \G)=\chi(\P,\wt{\G})$. Let $\wt{x}\in\wt{X}$ be the point with image $x=\pi(\wt{x})$. The constructions of $\wt{F}, \P$ provide the following equations of Chern characters:
\[
\begin{array}{rl}
\text{ch}(\wt{F})= & \text{ch}(\O_{\wt{X}}(h))+\text{ch}(\O_{\wt{X}}(-K_{\wt{X}}-h))=2-K_{\wt{X}}+\frac{1}{2}\O_{\wt{x}}\\
= & \text{ch}(\O_{\wt{X}})+\text{ch}(\O_{\wt{X}}(-K_{\wt{X}}))-2\text{ch}(\O_{\wt{x}}),
\end{array}
\]
\[
\begin{array}{rl}
\text{ch}(\P) = & \text{ch}(\O_{\wt{X}}(h))+\sum_{i=1}^4\text{ch}(\O_{\wt{X}}(e_i-K_{\wt{X}}-h))=5-3K_{\wt{X}}+\frac{5}{2}\O_{\wt{x}}\\
= & 2\text{ch}(\O_{\wt{X}})+3\text{ch}(\O_{\wt{X}}(-K_{\wt{X}}))-5\text{ch}(\O_{\wt{x}})
\end{array}
\]
Hence, Hirzebruch-Riemann-Roch implies that $\chi(F, \G)=\chi(\wt{F}, \wt{\G})=\chi(\O_{\wt{X}}, \wt{\G})+\chi(\O_{\wt{X}}(-K_{\wt{X}}), \wt{\G})-2\chi(\O_{\wt{x}}, \wt{\G})$. Applying the adjunction again yields the result and the argument for the second equation is similar.
\end{proof}
\begin{lemma}\label{semistable}
Let $\G$ be a semistable sheaf on $X$ whose Hilbert polynomial is $h_d(t)$ for $d=2, 3$.

(i) If $d=2$, then $\G\cong F$;

(ii) If $d=3$, then $\G\cong \pi_*\O(D)$ with $D\in\{h, e_i-K_{\wt{X}}-h, 1\leqslant i \leqslant 4\}$.
\end{lemma}
\begin{proof}
Note that $\Q$ is a vector bundle of rank $5$ with $h_{\Q}=5h_3$ and it is constructed from iterated extensions of stable sheaves $\pi_*\O(D)$ with $D\in\{h, e_i-K_{\wt{X}}-h, 1\leqslant i \leqslant 4\}$. By Lemma \ref{ss}, $\Q$ is semistable. Moreover, as rank $1$ sheaf, $\O_X$ is stable with $h_{\O_X}(t)=\frac{5}{2}t^2+\frac{5}{2}t+1$.

(i) $d=2$: First we show that $\G\in\A_2$. By Lemma \ref{cg}, we need to prove that $\Ext^*(\G,\O_X)=\Ext^*(\Q,\G)=0$. Since $h_3(t)>\frac{1}{2}h_2(t)>h_{\O_X}(t)$, by semistability,
\[\Hom(\Q,\G)=\Hom(\G,\O_X)=0.\]
Since $\frac{1}{2}h_2(t)>h_3(t-1)$ and $h_{\O_X}(t)>\frac{1}{2}h_2(t-1)$, 
\[\Hom(\G, \Q(K_X))=\Hom(\O_X, \G(K_X))=0.\]
By Serre Duality \cite[Proposition 2.6]{kuzdp6},
\[\Ext^2(\Q,\G)=\Ext^2(\G, \O_X)=0.\]
Note that $\Ext^i(\G, \O_X)=H^{2-i}(\G(K_X))^*$ and $\Ext^i(\Q, \G)=H^i(\Q^*\otimes\G)$ where $()^*$ is the dual. Hence, the Ext groups are zero for $i>2$. To show that $\G\in\A_2$, it remains to see that 
\[\chi(\G, \O_X)=\chi(\G(K_X))=h_2(-1)=0\] 
and 
\[\chi(\Q,\G)=2\chi(\G)+3\chi(\G(K_X))-5\cdot\mathrm{rank}(\G)=2h_2(0)+3h_2(-1)-5\cdot 2=0,\] 
which follows from the above lemma and the fact that the leading coefficient of $h_{\G}=h_2$ is $\frac{5}{2}\mathrm{rank}(\G)$.

Since $\G\in\A_2=\langle F\rangle=D^b(k)$ and as a pure sheaf, $\G$ is concentrated in degree $0$, $\G$ is a direct sum of $F$. Thus, $h_{\G}=h_F=h_2$ implies that $\G\cong F$. The proof for (ii) is similar.
\end{proof}
With the preparation of lemmas above, the same proof in \cite[Theorem 4.5]{kuzdp6} gives
\begin{thm}\label{ms}
Let $\M_d(X), d\in\{2,3\}$ be the moduli spaces of Gieseker semistable sheaves on $X$ with Hilbert polynomials $h_d(t)$ with respect to $-K_X$.  Then $\M_d(X)$ are fine moduli spaces. Moreover,

(i) $\M_2(X)\cong\mathrm{Spec(k)}$ and the vector bundle $F$ is the universal family;

(ii) $\M_3(X)\cong Z$ as in Theorem \ref{ssod} and the vector bundle $\Q$ is the universal family.
\end{thm}
%--------Global Generation-----------
\section{Global Generation}\label{global generation}
We use the same notation as section \ref{dcqdp} and prove that the rank $2$ vector bundle $F$ is globally generated. First, we provide a useful vanishing lemma.
\begin{lemma}\label{r1van}
Let $V$ be a surface with an isolated singular point $v$, which is of $A_n$ type. Let $f: \wt{V}\to V$ be the minimal resolution and $E=f^{-1}(v)$ be the exceptional locus. Then $E=E_1+\dots+E_n$ is a chain of $(-2)$-curves. Let $\O(D)$ be an invertible sheaf on $\wt{V}$ with degrees $d_i= D\cdot E_i, 1\leqslant i\leqslant n$. If for some $l\in \{1, \dots, n\}, d_l\geqslant -1$ and $d_i\geqslant 0$ for $i\neq l$, then $R^1f_*\O(D)=0$. 
\end{lemma}
\begin{proof}
For $p\geqslant 1$, define $E_{(p)}=\wt{V} \times_V \mathrm{Spec}(\O_v/m_v^p)$ where $m_v$ is the maximal ideal of the local ring $\O_v$ at the point $v$. By theorem of formal functions, $H^1(E_{(p)}, \O_{E_{(p)}}(D))=0$ for all $p\geqslant 1$ implies $R^1f_*\O(D)=0$. By Theorem $4$ in \cite{aratsing}, one has $E_{(p)}=pE=pE_1+\dots+pE_n$. We will prove the vanishing of $H^1(E_{(p)},\O_{E_{(p)}}(D))$ inductively.

First let $l=1$. It is clear that $H^1(E_1, \O_{E_1}(D))=0$ and assume that $H^1(W, \O_W(D))=0$ for $W=mE_1+\dots+mE_i+(m-1)E_{i+1}+\dots+(m-1)E_n, m\geqslant 1, n\geqslant i\geqslant 1$. Let $Z=W+E_{i+1}$ and identify $E_{n+1}=E_1$. We have the short exact sequence
\[0\to \O_{E_{i+1}}(D-W)\to \O_Z(D)\to \O_W(D)\to 0.\]
There are $3$ different cases: if $1\leqslant i\leqslant n-2$, then $E_{i+1}\cdot(D-W)=E_{i+1}\cdot D-1\geqslant -1$; if $i=n-1$, then $E_{i+1}\cdot(D-W)=E_{i+1}\cdot D+m-2\geqslant -1$; if $i=n$, then $E_{i+1}\cdot(D-W)= E_{i+1} \cdot D+m\geqslant 0$.
Therefore, $H^1$ of the first sheaf is $0$ and we have $H^1(Z, \O_Z(D))=0$.

In the argument above, the vanishing of $H^1$ is proved by adding divisors in the order of $E_1, \dots, E_n$. In the general case, the same proof applies by changing the order to $E_l,\dots, E_n, E_{l-1}, \dots, E_1$. 
\end{proof}
\begin{lemma}
Let $\pi: \wt{X}\to X$ be the minimal resolution. Then 

(i) $\pi_*\O_{\wt{X}}(h)$ is globally generated, $R^1\pi_*\O_{\wt{X}}(h)=0$ and 
\[
h^i(\O(h))=
\left\{
\begin{array}{ll}
3, &i=0\\
0, & i\neq 0\\
\end{array}
\right.
;\]
(ii) $R^1\pi_*\O(-K_{\wt{X}}-h)=0, \pi_*\O(-K_{\wt{X}}-h)$ is globally generated and
\[
h^i(\O(-K_{\wt{X}}-h))=
\left\{
\begin{array}{ll}
2, &i=0\\
0, & i\neq 0\\
\end{array}
\right.
;\]
(iii) $F$ and thus $\wt{F}=\pi^*F$ are globally generated and\[
h^i(X, F)=h^i(\wt{X}, \wt{F})=
\left\{
\begin{array}{ll}
5, &i=0\\
0, & i\neq 0\\
\end{array}
\right.
.\] 
Moreover, $\det(\wt{F})=\O(-K_{\wt{X}})$ and $\det(F)=\O(-K_X)$.
\end{lemma}
\begin{proof}
(i) Let $f: \wt{X}\to \mathbb{P}^2$ be the blow up of $4$ points $x_1,\dots, x_4$. Pulling back the Euler sequence on $\mathbb{P}^2$ along $f$, we get
\[0\to f^*\Omega_{\mathbb{P}^2}(h)\to \O_{\wt{X}}^3\to \O_{\wt{X}}(h)\to 0.\]
The restriction of $f^*\Omega_{\mathbb{P}^2}(h)$ on $\Delta_{ij}$ is trivial and on $\Delta_{ijl}$ is $\O\oplus\O(-1)$. Therefore, by Lemma \ref{r1van}, $R^1\pi_*f^*\Omega_{\mathbb{P}^2}(h)=0$ and the Euler sequence implies that $\pi_*\O(h)$ is globally generated. The vanishing of $R^1\pi_*\O_{\wt{X}}(h)$ is similar and the computation of $h^i$ is straightforward.

(ii) Lemma \ref{r1van} implies that $R^1\pi_*\O(-K_{\wt{X}}-h)=0$. Let $l_{ij}\in |h-e_i-e_j|$ be the $(-1)$-curve and $l_i\in |h-e_i|$ be the strict transform of the line passing through the point $x_i$ (if they exist). Note that $h^0(\O(l_i))=2, h^p(\O(l_i))=0$ when $p\neq 0$ and $h^0(\O(l_{ij}))=1, h^p(\O(l_{ij}))=0$ when $p\neq 0$ . The computation of $h^i(\O(-K_{\wt{X}}-h))$ depends on the singular type of $X$:

For (I.1) (II.1) (III.1), use
\[0\to \O(h-e_3-e_4)=\O(l_{34})\to\O(-K_{\wt{X}}-h)\to \O_{l_{12}}\to 0.\]

For (IV.1), use further
\[0\to \O(h-e_2-e_4)\to\O(h-e_3-e_4)\to \O_{\Delta_{23}}(-1)\to 0,\]
\[0\to \O(l_{14})\to \O(h-e_2-e_4)\to \O_{\Delta_{12}}(-1)\to 0\]
which imply that $h^i(\O(h-e_3-e_4))=h^i(\O(l_{14}))$. Similarly, for the case (V.1), we have $h^i(\O(h-e_3-e_4))=h^i(\O(l_{12}))$.

For (I.2) (II.2) (IV.2), use
\[0\to \O(h-e_4)=\O(l_4)\to \O(-K_{\wt{X}}-h)\to\O_{\Delta_{123}}(-1)\to 0.\]

For (III.2), use further
\[0\to \O(l_3)\to\O(h-e_4)\to \O_{\Delta_{34}}(-1)\to 0\]
which implies that $h^i(\O(h-e_4))=h^i(\O(l_3))$. Similarly, for the case (V.2), we have $h^i(\O(h-e_4))=h^i(\O(l_1))$.

For (II.3), use
\[0\to\O(h-e_2)\to \O(-K_{\wt{X}}-h)\to \O_{\Delta_{134}}(-1)\to 0,\]
\[0\to \O(l_1)\to \O(h-e_2)\to\O_{\Delta_{12}}(-1)\to 0.\]

Similarly for (IV.3), use
\[0\to\O(h-e_3)\to \O(-K_{\wt{X}}-h)\to \O_{\Delta_{124}}(-1)\to 0\]
and $h^i(\O(h-e_3))=h^i(\O(l_1))$.

For cases (*.1), we deduce that $\O(-K_{\wt{X}}-h)$ is base-point free because it has sections:
\begin{description}
\item[(I.1)] $l_{12}+l_{34}$ and $l_{13}+l_{24}$;
\item[(II.1)] $l_{12}+l_{34}$ and $l_{13}+l_{14}+\Delta_{12}$;
\item[(III.1)] $l_{12}+l_{34}$ and $2l_{13}+\Delta_{12}+\Delta_{34}$;
\item[(IV.1)] $l_{12}+l_{14}+\Delta_{12}+\Delta_{23}$ and $C_1$;
\item[(V.1)] $2l_{12}+\Delta_{12}+2\Delta_{23}+\Delta_{34}$ and $C_2$.
\end{description}
Here $C_1, C_2\in |-K_{\wt{X}}-h|=|2h-e_1-e_2-e_3-e_4|$ are conics. The kernel of the evaluation map $\O_{\wt{X}}^2\to \O(-K_{\wt{X}}-h)$ is reflexive and thus an invertible sheaf. Therefore, we have
\[0\to \O(K_{\wt{X}}+h)\to \O_{\wt{X}}^2\to\O(-K_{\wt{X}}-h)\to 0.\]
Moreover, $R^1\pi_*\O(K_{\wt{X}}+h)=0$ by Lemma \ref{r1van} and thus $\pi_*\O(-K_{\wt{X}}-h)$ is globally generated.

For the rest cases, from the computation above, we have $\pi_*\O(-K_{\wt{X}}-h)=\pi_*\O(l_i)$ for some $i$. Since $\O(l_i)$ is base-point free, the short exact sequence coming from extending the evaluation map
\[0\to \O(-l_i)\to\O_{\wt{X}}^2\to\O(l_i)\to 0\]
plus $R^1\pi_*\O(-l_i)=0$ imply that $\pi_*\O(l_i)$ is globally generated.

(iii) It is clear that $\det(\wt{F})=\O(-K_{\wt{X}})$ and thus $\det(F)=\pi_*\pi^*\det(F)=\pi_*\det(\wt{F}) = \O(-K_X)$. The rest follows from (i)(ii).
\end{proof}
%-----------Galois Descent----------
\section{Galois Descent}\label{galois descent}
Let $k$ be an arbitrary field with the separable closure $k_s$. Let $\Gamma=\Gal(k_s/k)$ be the absolute Galois group. Let $Y$ be a projective variety over $k$. Denote the base extension to $k_s$ by $Y_{k_s}= Y\times_k k_s$. Fix a projective variety $W$ over $k$. We say $Y$ is a \textit{twisted form} of $W$ if there is a $k_s$-isomorphism $\phi: W_{k_s}\to Y_{k_s}$. Twisted forms of $W$ are classified by the first Galois cohomology $H^1(k, \mathrm{Aut}_{k_s}(W_{k_s}))=H^1(\Gamma, \mathrm{Aut}_{k_s}(W_{k_s}))$ \cite[III \S1.3]{galoiscoh}.

In detail, the correspondence is given as follows. $W_{k_s}$ and $Y_{k_s}$ have a natural Galois action with $\Gamma$ acting on the factor $k_s$. For $\sigma\in\Gamma$, define $a_{\sigma}= \phi^{-1} \circ \sigma \circ \phi \circ \sigma^{-1}$. Then $a_{\sigma}\in\mathrm{Aut}_{k_s}(W_{k_s})$ is a $1$-cocycle, i.e. $a_{\sigma\tau}=a_{\sigma}{}^{\sigma}a_{\tau}$ ($\Gamma$ acts on $\mathrm{Aut}_{k_s}(W_{k_s})$ by inner automorphisms). The form $Y$ corresponds to the cocycle class $[a_{\sigma}]$. A different choice of $\phi$ produces the same cocycle class. Conversely, for a $1$-cocycle class, choose a $1$-cocycle representative $a_{\sigma}\in\mathrm{Aut}_{k_s}(W_{k_s})$. Define an associated twisted $\Gamma$-action on $W_{k_s}$ by sending $(\sigma, x)\in \Gamma\times W_{k_s}$ to $a_{\sigma}(\sigma(x))$. Since $a_{\sigma}$ is a $1$-cocycle, $a_{\sigma\tau}(\sigma\tau(x))=a_{\sigma}\sigma(a_{\tau})(\sigma\tau(x))=a_{\sigma}\sigma(a_{\tau}\tau(x))$. Thus, we indeed obtain a $\Gamma$-action. Taking the invariants of this twisted action, we obtain $Y=(W_{k_s})^{\Gamma}$ as a twisted form of $W$ over $k$. A different choice of the cocycle representative produces an isomorphic form.

Let $A$ be a central simple $k$-algebra. Write $\mathrm{SB}_r(A)$ for the generalized Severi-Brauer variety, which by definition is the variety of right ideals of dimension $r\deg A$ over $k$. It is a twisted form of Grassmannians because for a vector space $V$, one has $\mathrm{SB}_r(\mathrm{End}(V))\cong\Gr(r, V)$. For more details, see \cite[I \S1]{boi}. 
\begin{lemma}\label{ggm}
Let $Y$ be a projective variety over $k$. Let $l$ be a Galois extension of $k$ and $G=\Gal(l/k)$ be its Galois group. Let $Y_l=Y\times_k l$ be the field extension equipped with the natural $G$-action. The following are equivalent:

(i) There exists a morphism $f: Y\to \mathrm{SB}_r(A)$ over $k$ where $A$ is a central simple $k$-algebra that splits over $l$, i.e. $A\otimes_k l=\mathrm{End}(W)$ for some vector space $W$ over $l$.

(ii) There exists a $G$-invariant globally generated vector bundle $N$ of rank $r$ on $Y_l$ such that one can choose, for each $\sigma\in G$, an isomorphism $\phi_{\sigma}: {}^{\sigma}N\to N$ over $Y_l$ satisfying $ \phi_{\sigma}{}^{\sigma}\phi_{\tau}\phi_{\sigma\tau}^{-1}\in l^{\times}\subset\mathrm{Aut}_{Y_l}(N)$ for any $\sigma,\tau\in G$ (the inclusion is given by multiplying elements of $l^{\times}$). Here ${}^{\sigma}N$ denotes the pull-back of $N$ along $\sigma: Y_l\to Y_l$. 

Moreover, given (ii), if the global section $H^0(Y_l, N)$ has dimension $n$, then $A$ can be chosen to have degree $n$ and $N$ is the pull-back of a vector bundle on $Y$ if and only if the Brauer class $[A]\in\mathrm{Br}(k)$ is trivial.
\end{lemma}
\begin{proof}
Given $f: Y\to \mathrm{SB}_r(A)$, after field extension, we obtain a $G$-invariant morphism $f\times_k l: Y_l\to \Gr(r, W)$ where $W$ is a vector space over $l$.  Let $\R$ be the universal subbundle of $\Gr(r,W)$ and denote its dual by $\R^*$. Then $H^0(\Gr(r, W), \R^*)=W^*$ and $\Hom(\R^*, \R^*)=l$. We check that $N=(f\times_k l)^*(\R^*)$ is the required vector bundle in (ii).

Clearly, the vector bundle $N$ is globally generated of rank $r$. There is a natural map $\mathrm{Aut}(W)\to \mathrm{Aut}(\Gr(r, W))$ which factors through $\mathrm{PGL}(W)$ and the form $\mathrm{SB}_r(A)$ corresponds to the image of $[A]$ under the induced map $H^1(G,\mathrm{PGL}(W))\to H^1(G, \mathrm{Aut}(\Gr(r, W)))$. Since $\R^*$ is invariant under $\mathrm{Aut}(W)$, we have $\R^*$ and thus $N$ are $G$-invariant. Therefore, there are isomorphisms $\psi_{\sigma}: {}^{\sigma}\R^*\to \R^*$ for $\sigma\in G$. Because $\Hom(\R^*, \R^*)=l$, one has $\psi_{\sigma}{}^{\sigma}\psi_{\tau}\psi_{\sigma\tau}^{-1}\in l^{\times}$ for $\sigma, \tau\in G$. The isomorphisms $\phi_{\sigma}$ can be chosen to be the pull-backs of $\psi_{\sigma}$ along $f\times_k l$.

Conversely, given (ii), isomorphisms $\phi_{\sigma}: {}^{\sigma}N\to N$ correspond to isomorphisms $\varphi_{\sigma}: N\to\sigma_*N$ and they induce $l$-automorphisms $b_{\sigma}=H^0(\varphi_{\sigma})$ on $V^*=H^0(Y_l, N)$. The condition ${}^{\tau}b_{\sigma}b_{\tau}b_{\tau\sigma}^{-1}=\phi_{\sigma}{}^{\sigma}\phi_{\tau}\phi_{\sigma\tau}^{-1}\in l^{\times}$ implies that elements $a_{\sigma}\in \mathrm{Aut}(\Gr(r, V))$ induced by $b_{\sigma}\in \mathrm{Aut(V^*)}$ form a $1$-cocycle. We equip $\Gr(r, V)$ with the twisted $G$-action associated to $a_{\sigma}$. Let $g: Y_l\to \Gr(r, V)$ be the morphism induced by the surjection $V^*\otimes_l \O_{Y_l}\to N$. By construction, it is $G$-invariant and thus descends to $f: Y\to \mathrm{SB}_r(A)$. Note that $A\otimes_k l=\mathrm{End}(V)$. Hence, the degree of $A$ is equal to the dimension of $V^*=H^0(Y_l,N)$.

Finally, $[A]\in H^1(G, \mathrm{PGL}(V))$ is the Brauer class induced by $b_{\sigma}$. If $N$ is the pull-back of a vector bundle on $Y$, then the isomorphisms $\phi_{\sigma}$ can be chosen such that $\phi_{\sigma}{}^{\sigma}\phi_{\tau}\phi_{\sigma\tau}^{-1}=1$. Thus, $[A]$ is trivial. On the other hand, if $[A]$ is trivial, then $A=M_n(k)$ and $\mathrm{SB}_r(A)=\Gr(r, n)$. Now regard $\R$ as the universal subbundle of $\Gr(r, n)$ over $k$. Then $N$ is the pull-back of $f^*(\R^*)$ on $Y$.
\end{proof}
Now let $X$ be a quintic del Pezzo surface over $k$ with rational Gorenstein singularities and let $\pi: \wt{X}\to X$ be its minimal resolution. Theorem 1 of \cite{cseps} states that every geometrically rational surface is separably split. In our case, it indicates that $\wt{X}_{k_s}$ is the blow-up of $\p^2$ at $4$ points as in \S\ref{qdp} and $X_{k_s}$ is obtained by contracting $(-2)$-curves on $\wt{X}_{k_s}$. Since the assumption of the base field $k$ being algebraically closed is only placed to make $X$ split, all previous results apply to $X_{k_s}$ as well. Recall that there is a rank $2$ vector bundle $F=\pi_{k_s*}\wt{F}$ on $X_{k_s}$ defined by (\ref{uex}) and (\ref{vb}).
\begin{lemma}
Vector bundles $\wt{F}$ and $F$ are Galois invariant.
\end{lemma}
\begin{proof}
Note that $F=\pi_{k_s*}\wt{F}$ and the map $\pi_{k_s}$ is Galois invariant. Thus, the Galois invariance of $F$ follows from that of $\wt{F}$. By the semiorthogonal decomposition (\ref{sodwdp}), we have $\langle\wt{F}\rangle=\wt{\A_2}={}^{\perp}\wt{\A_1}\cap\wt{\A_3}^{\perp}={}^{\perp}\O_{\wt{X}_{k_s}}\cap\wt{\A_3}^{\perp}$. The structure sheaf $\O_{\wt{X}_{k_s}}$ is certainly Galois invariant and thus it suffices to show that $\wt{\A_3}$ is Galois invariant. We will achieve this by proving that the set $\{\O(h), \O(e_i-K_{\wt{X}_{k_s}}-h), 1\leqslant i\leqslant 4\}$ is stable under $\mathrm{Aut}(\wt{X}_{k_s})$, i.e. automorphisms permute elements of the set. The action of $\mathrm{Aut}(\wt{X}_{k_s})$ on $\mathrm{Pic}(\wt{X}_{k_s})$ preserves $K_{\wt{X}_{k_s}}$ and inner product. By Theorem 23.9 of \cite{cubicforms}, it suffices to check that the set is stable under the Weyl group of the root system $R=K_{\wt{X}_{k_s}}^{\perp}\subset \mathrm{Pic}(\wt{X}_{k_s})\otimes_{\mathbb{Z}}\mathbb{R}$. It is straightforward to check that the reflections corresponding to simple roots $e_i-e_{i+1}, 1\leqslant i\leqslant 3, h-e_1-e_2-e_3$ permute the set.
\end{proof}
\begin{lemma}\label{gdf}
The vector bundle $F$ on $X_{k_s}$ descends to $X$. That is, $F$ is the pull-back of a vector bundle on $X$ along the natural projection $p:X_{k_s}\to X$.
\end{lemma}
\begin{proof}
From the previous lemma, $F$ is Galois invariant. Moreover, $\Hom(F,F)=k_s$ because $F$ is an exceptional object. Thus, $F$ is a vector bundle satisfying Lemma \ref{ggm} (ii) and we have a morphism $f: X\to \mathrm{SB}_2(A)$ where $A$ is a central simple $k$-algebra of degree $5$. Since $X$ is rational (mentioned in the introduction), $\mathrm{SB}_2(A)(k)\neq\emptyset$. \cite[Proposition 1.17]{boi} indicates that the index of $A$ divides $2$. Being of degree $5$, it forces $A$ to be split and thus $F$ descends.
\end{proof}
\begin{lemma}\label{lsgr}
Let $X$ be a quintic del Pezzo surface over $k$ with rational Gorenstein singularities. Let $N$ be a rank $2$ vector bundle on $X$ with $\det(N)=\O_X(-K_X)$ and a surjection map $\O_X^{\oplus 5}\to N$. We have the following commutative diagram:
\begin{center}
\begin{tikzcd}
X\arrow{r}{f}\arrow[hook]{d}{g} &\Gr(2, 5)\arrow[hook]{d}{h}\\
\p^5\arrow[hook]{r}{i} & \p^9\\
\end{tikzcd}
\end{center}
where $f, g$ are induced by the surjection $\O_X^{\oplus 5}\to N$ and the linear system of $\det(N)$ respectively and $h$ is the Pl\"ucker embedding. Then $f$ is injective and $X=\Gr(2,5)\cap\p^5\subset \p^9$.  By symmetry of $\Gr(2,5)\cong\Gr(3,5)$, the same result holds if $N$ is of rank $3$. 
\end{lemma}
\begin{proof}
The map $f$ is injective because $i\circ g$ is. Let $J$ be the ideal of $X$ in $\p^9$. The short exact sequences
\[0\to J\to \O_{\p^9}\to \O_X\to 0\]
twisted with $\O_{\p^9}(1), \O_{\p^9}(2)$ imply $h^0(J(1))=4, h^0(J(2))=39$. Moreover, among the $39$-dimensional family of quadrics containing $X$, $34$-dimension are from degenerate quadrics. On the other hand, $\Gr(2,5)$ is the intersection of $5$ nondegenerate quadrics in $\p^9$. By \cite[Theorem 4.4(i)]{hwgdp}, $X$ is the intersection of $5$ quadrics in $\p^5$. This implies that $X=\Gr(2,5)\cap\p^5$.
\end{proof}

In $\S$\ref{hpdapproach}, we provide two constructions for realizing a quintic del Pezzo surface $X$ as a linear section of $\Gr(2,5)$. By the theory of HPD, $D^b(X)$ can be described explicitly given that the dual linear section of $X$ has the expected dimension. We verify below that the dual linear section does have the expected dimension.

Let $V_5$ be a $5$-dimensional vector space over $k$ and $V_5^*$ be the dual vector space. Let $W_6$ be a $6$-dimensional subspace of $\bigwedge^2 V_5$ and define $W_6^{\perp}:=\ker(\bigwedge^2 V_5^*\to W_6^*)$ as its orthogonal. 
\begin{lemma}\label{dualsection}
Set $X:= \Gr(2, V_5)\cap \p(W_6)$ inside $\p(\bigwedge^2 V_5)$ and $Y:= \Gr(2, V_5^*)\cap\p(W_6^{\perp})$ inside $\p(\bigwedge^2 V_5^*)$. If $\dim X=2$ and $X$ has rational Gorenstein singularity, then $\dim (Y)=0$.
\end{lemma}
\begin{proof}
The assumption implies that $X$ is a rational Gorenstein quintic del Pezzo surface such that $X\hookrightarrow \p(W_6)$ is the anticanonical embedding. By \cite[Proposition 4.2(i)]{hwgdp}, there exists a $5$-dimensional $W_5\subset W_6$ such that the hyperplane section $C:= \Gr(2, V_5)\cap \p(W_5)$ of $X$ is a smooth elliptic curve. Thus, the dual linear section $C':= \Gr(2, V_5^*)\cap \p(W_5^{\perp})$ is also a smooth elliptic curve by \cite[Proposition 2.24]{dkgm}. Since $Y$ is a hyperplane section of $C'$, $\dim(Y)=0$.
\end{proof}
%---------Quintic Del Pezzo Fibrations----------------
\section{Quintic Del Pezzo Fibrations}\label{qdpf}
Let $k$ be an arbitrary field and $\bar{k}$ be the algebraic closure.
\begin{defn}\label{qdpfdef}
Let $f:\mathcal{X}\to S$ be a flat morphism over $k$. The map $f$ is a \textit{quintic del Pezzo fibration} if for any point $s\in S$, the fiber $\mathcal{X}_s$ is a quintic del Pezzo surface with rational Gorenstein singularities. Denote the geometric fiber over the point $s\in S$ by $\mathcal{X}_{\bar{s}}$.
\end{defn}
As is the case in Appendix B of \cite{kuzdp6}, we can consider the moduli stack $\mathfrak{DP}_5$ of singular quintic del Pezzo surfaces. It is the fibered category over the category of schemes over a field $k$ whose fiber over a $k$-scheme $S$ is the groupoid of all quintic del Pezzo fibrations over $S$ in the above sense. Theorem B.1 in \textit{loc. cit.} also proves the following fact.
\begin{prop}
The moduli stack $\mathfrak{DP}_5$ is a smooth Artin stack of finite type over $k$.
\end{prop}
Note that combining with the base change of semiorthogonal decompositions (Proposition \ref{bc}), it suffices to prove Theorem \ref{main} when the base $S$ is a smooth variety.
%---------Moduli Space Approach-------------
\subsection{Moduli Space Approach}\label{msapproach}
Let $\M_d(\mathcal{X}/S), d\in\{2,3\}$ be the relative moduli space of semistable sheaves on fibers of $f:\mathcal{X}\to S$ with Hilbert polynomial $h_d(t)$ defined by (\ref{hilb}). Comparing with Theorem \ref{ms}, similar results hold for relative moduli spaces as well.
\begin{prop}\label{msf}
For $d\in\{2,3\}, \M_d(\mathcal{X}/S)$ are fine moduli spaces. Let $\mathcal{E}_d$ be the universal families of $\M_d(\mathcal{X}/S)$. Then

(i) $\M_2(\mathcal{X}/S)\cong S$ and $\mathcal{E}_2 |_{\mathcal{X}_{\bar{s}}}$ is the vector bundle $F$;

(ii) $g: \M_3(\mathcal{X}/S)\cong \mathcal{Z}\to S$ is flat and finite of degree $5$. View $\mathcal{E}_3$ as a sheaf on $\mathcal{X}$ via the finite morphism $\mathcal{X}\times_S \mathcal{Z}\to\mathcal{X}$. The geometric fiber of $g$ is the scheme $Z$ and $\mathcal{E}_3 |_{\mathcal{X}_{\bar{s}}}$ is the vector bundle $\Q$.

Moreover, $\mathcal{E}_d$ is flat over $\M_d(\mathcal{X}/S)$ and is a locally free sheaf over $\mathcal{X}$. In particular, $\mathcal{E}_d$ has finite Tor-amplitude over $\M_d(\mathcal{X}/S)$ and finite Ext-amplitude over $\mathcal{X}$.
\end{prop}
\begin{proof}
As with the sextic del Pezzo case \cite[Proposition 5.3]{kuzdp6}, there exist the coarse moduli spaces $\M_d(\mathcal{X}/S)\cong\mathcal{Z}_d$ and quasi-universal families $\mathcal{E}_d$ on the fiber products $\mathcal{X} \times_S \mathcal{Z}_d$. The Brauer obstructions $\beta_d\in\mathrm{Br}(\mathcal{Z}_d)$ for the coarse moduli spaces to be fine and the quasiuniversal familes to be universal have orders dividing the greatest common divisor of the values of $h_d(t)$.

Note that by construction, the (coarse) moduli spaces, (quasi-)universal families and Brauer obstructions are compatible with the base change. The claims for the geometric fiber of $g$ and $\mathcal{E}_d |_{\mathcal{X}_{\bar{s}}}$ follow from Theorem \ref{ms}. To prove (i)(ii), it remains to show that the obstructions $\beta_d$ are trivial and $g$ is flat. Clearly, the g.c.d of values of $h_3(d)$ is $1$. Thus, $\M_3(\mathcal{X}/S)$ is a fine moduli space. Since the fibers of $g$ are of the same length ($=5$), i.e. $g_*\O_{\mathcal{Z}}$ is locally free, $\mathcal{Z}$ is Cohen Macauley by \cite[Corollary 18.17]{eca} and $g$ is flat by \cite[Theorem 23.1]{mcrt}.

Let $\eta$ be the generic point of $S$. Because $S$ is regular integral, the restriction map $j: \mathrm{Br}(S)\to\mathrm{Br}(k(\eta))$ is injective. By Lemma \ref{gdf}, the vector bundle $F$ on $\mathcal{X}_{\bar{\eta}}$ descends to $\mathcal{X}_{\eta}$. The Brauer obstruction $j(\beta_2)$ is trivial. Hence, $\beta_2$ is trivial and $\M_2(\mathcal{X}/S)$ is a fine moduli space as well.

The rest follows from \cite[Lemma 5.7]{kuzdp6}. The original argument for the locally freeness of $\mathcal{E}_d$ was not clear to us. We give a revised proof as follows. Let $s\in S, x\in\mathcal{X}_s$ be points. Denote inclusions by $i: x\overset{r}{\hookrightarrow}\mathcal{X}_s\overset{t}{\hookrightarrow}\mathcal{X}$. Because $\M_d(\mathcal{X}/S)$ are flat over $S$, $\mathcal{E}_d$ are flat over $S$ and $t^*\mathcal{E}_d=L_0t^*\mathcal{E}_d$ are corresponding vector bundles in (i)(ii). Then $i^*\mathcal{E}_d=r^*t^*\mathcal{E}_d=L_0r^*L_0t^*\mathcal{E}_d$ implies that $L_0i^*\mathcal{E}_d$ are vector spaces of dimension independent of $x$ and $L_1i^*\mathcal{E}_d=0$. Hence, $\mathcal{E}_d$ are locally free over $\mathcal{X}$.
\end{proof}
Now we are ready to give the proof of the main theorem.
\begin{proof}[Proof of Theorem \ref{main}]
The proof is the same as Theorem $5.2$ in \cite{kuzdp6}. The rough idea is that $\langle D^b(S), D^b(S), D^b(\mathcal{Z})\rangle$ is an $S$-linear full semiorthogonal collection for $D^b(\mathcal{X})$ because it is so for each geometric fiber by Theorem \ref{ssod}. In particular, the embeddings of components are given by Fourier-Mukai functors with kernels $\O_{\mathcal{X}}, \mathcal{E}_2,\mathcal{E}_3$, which have finite Tor-amplitudes over $S, S, \mathcal{Z}$ respectively and finite Ext-amplitudes over $\mathcal{X}$ by Proposition \ref{msf}.  Hence, the finiteness of cohomological amplitudes of the projection functors and the compatibility with the base change follow from Proposition \ref{bc}.
\end{proof}
%--------Homological Projective Duality Approach------------
\subsection{Homological Projective Duality Approach}\label{hpdapproach}
Given a flat family of quintic del Pezzo surfaces $f:\mathcal{X}\to S$ as in Definition \ref{qdpfdef}, $\mathcal{X}$ can be realized as a linear section of a Grassmannian bundle over $S$. We give two such constructions as follows.
%------First Construction------------
\subsubsection{First Construction}
Let us first go back to the case of a single quintic del Pezzo surface. Let $X$ be a quintic del Pezzo surface with rational singularities over an arbitrary field $k$. Sections \ref{global generation},\ref{galois descent} indicate that  on $X$ there is a rank $2$ globally generated vector bundle with determinant $\O_X(-K_X)$ and global sections of dimension $5$ (base changed to the algebraic closure of $k$, this is the vector bundle $F$ defined by (\ref{vb})). Hence, $X=\p^5\cap\Gr(2,5)$ by Lemma \ref{lsgr}. This construction can be generalized to a family because this vector bundle has another unique property: it is the only (semi)stable sheaf on $X$ whose Hilbert polynomial is $h_2(t)$ (Lemma \ref{semistable}(i)). Therefore, the universal family $\mathcal{E}_2$ of the fine moduli space $\mathcal{M}_2(\mathcal{X}/S)$ induces
\begin{equation}\label{first}
\mathcal{X}=\p_S((f_*\omega_{\mathcal{X}/S}^{-1})^*)\times_S \Gr_S(2, (f_*\mathcal{E}_2)^*)\subset \p_S(\bigwedge^2(f_*\mathcal{E}_2)^*)
\end{equation}
where $\omega_{\mathcal{X}/S}^{-1}$ is the relative anticanonical sheaf. Define $\L^\perp$ as the kernel
\[0\to \L^\perp\to \bigwedge^2 f_*\mathcal{E}_2\to f_*\omega_{\mathcal{X}/S}^{-1}\to 0\]
and $\mathcal{Z}'$ as the corresponding dual linear section of $\Gr_S(2, f_*\mathcal{E}_2)$, i.e.
\[\mathcal{Z}'=\p_S(\L^\perp)\times_S \Gr_S(2,f_*\mathcal{E}_2)\subset \p_S(\bigwedge^2f_*\mathcal{E}_2).\]
\begin{thm}\label{main2}
Let $k$ be the base field and $\mathrm{char}(k)\neq 2,3$. Let $f:\mathcal{X}\to S$ be a quintic del Pezzo fibration as in Definition \ref{qdpfdef} and denote by $g':\mathcal{Z}'\to S$ the fibration constructed above. Then we obtain a semiorthogonal decomposition same as Theorem \ref{main}. In particular, $g'=g: \mathcal{Z}'\cong\mathcal{Z}\to S$.  
\end{thm}
\begin{proof}
By construction and Lemma \ref{dualsection}, each fiber of $g':\mathcal{Z}'\to S$ has length $5$ and thus $g'$ is flat and finite of degree $5$. We will use the special version of Homological Projective Duality introduced in \cite{kuzhs}, in particular, the relative version of Example 6.1 in \textit{loc.cit.}. In order to do so, we need a Lefschetz type semiorthogonal decomposition, which is given by Proposition \ref{lefschetzgr} and a relative version of this decomposition also holds by arguments in \cite[\S3]{orlproj}. This implies that we have an $S$-linear semiorthogonal decomposition
\begin{equation}\label{sodhpd}
\begin{array}{rl}
D^b(\mathcal{X}) &= \langle D^b(\mathcal{Z}'), f^*D^b(S)\otimes\omega_{\mathcal{X}/S}^{-1}, f^*D^b(S)\otimes\mathcal{E}_2\otimes \omega_{\mathcal{X}/S}^{-1}\rangle\\
&= \langle f^*D^b(S)\otimes\O_{\mathcal{X}}, f^*D^b(S)\otimes\mathcal{E}_2, D^b(\mathcal{Z}')\rangle\\
&= \langle D^b(S), D^b(S), D^b(\mathcal{Z}')\rangle.
\end{array}
\end{equation}
The second equality is obtained by applying the Serre functor $-\otimes \omega_{\mathcal{X}/S}[\dim\mathcal{X}-\dim S]$ to the last two components. The embeddings of the components are given by Fourier-Mukai functors with kernels $\O_{\mathcal{X}}$, $\mathcal{E}_2$ and $\mathcal{E}'$ respectively. Moreover, $\mathcal{E}'$ has finite Tor-ampiltude over $\mathcal{Z}'$ and finite Ext-amplitude over $\mathcal{X}$. Therefore, the finiteness of cohomological amplitudes and compatibility of the base change follow by the same reason as before.

Finally, the decomposition (\ref{sodhpd}) is exactly the same as the one obtained by the moduli space approach. Comparing these two decompositions, we notice that there is an $S$-linear equivalence $D^b(\mathcal{Z}')\simeq {}^\perp\langle f^*D^b(S), f^*D^b(S)\otimes\mathcal{E}_2\rangle\simeq D^b(\mathcal{Z})$. By Morita equivalence, one gets $\mathcal{Z}'\cong\mathcal{Z}$ over $S$ and thus the kernels for the embedding functors are also isomorphic, i.e. $\mathcal{E}'\cong\mathcal{E}_3$. 
\end{proof}
%--------Second Construction------------
\subsubsection{Second Construction}\label{construction2}
Let $X$ be a quintic del Pezzo surface with rational Gorenstein singularities over an arbitrary field $k$. Let $X\to\p^5$ be the anticanonical embedding and $I$ be the ideal of $X$ in $\p^5$. One can compute that $I(2)$ is globally generated with $h^0(\p^5, I(2))=5$. Therefore, we have a rank $3$ vector bundle $N^*_{X/\p^5}(2)=I/I^2(2)$ with determinant $\O_X(-K_X)$ and a surjection $\O_X^{\oplus 5}\to N_{X/\p^5}^*(2)$. Define $F'$ as the cokernel
\begin{equation}\label{normalseq}
0\to N_{X/\p^5}(-2)\to \O_X^{\oplus 5}\to F'\to 0.
\end{equation}
Then $F'$ is a rank 2 vector bundle with determinant $\O_X(-K_X)$. Again, we obtain $X=\p^5\cap \Gr(2,5)$ by $F'$ and this construction can be generalized to a family as well. For a family $f:\mathcal{X}\to S$, let $\mathcal{I}$ be the ideal sheaf of the anticanonical embedding $\mathcal{X}\hookrightarrow \p_S((f_*\omega_{\mathcal{X}/S}^{-1})^*)$ over S and $N_{\mathcal{X}/\p_S}$ be the normal bundle. Let $m: \p_S((f_*\omega_{\mathcal{X}/S}^{-1})^*)\to S$ be the projection. Then we have the short exact sequence
\begin{equation}\label{normalseqfamily}
0\to N_{\mathcal{X}/\p_S}(-2)\to f^*(m_*\mathcal{I}(2))^*\to \mathcal{F}'\to 0
\end{equation}
where $\O(1)=\omega_{\mathcal{X}/S}^{-1}$ and $\mathcal{F}'$ is the rank $2$ vector bundle defined as the cokernel. Similarly, we have the linear section structure for $\mathcal{X}$ and the dual linear section $g'':\mathcal{Z}''\to S:$
\begin{equation}\label{second}
\mathcal{X}=\p_S((f_*\omega_{\mathcal{X}/S}^{-1})^*)\times_S \Gr_S(2, m_*\mathcal{I}(2))\subset \p_S(\bigwedge^2 m_*\mathcal{I}(2)),
\end{equation}
\[0\to \L'^\perp \to\bigwedge^2 (m_*\mathcal{I}(2))^*\to f_*\omega_{\mathcal{X}/S}^{-1}\to 0, \]
\[\mathcal{Z}''=\p_S(\L'^\perp)\times_S \Gr(2, (m_*\mathcal{I}(2))^*)\subset \p_S(\bigwedge^2 (m_*\mathcal{I}(2))^*).\]
\begin{thm}\label{main3}
Let $k$ be the base field and $\mathrm{char}(k)\neq 2,3$. Let $f:\mathcal{X}\to S$ be a quintic del Pezzo fibration as in Definition \ref{qdpfdef} and denote by $g'':\mathcal{Z}''\to S$ the fibration constructed above. Then there is an $S$-linear semiorthogonal decomposition compatible with the base change
\begin{equation}\label{sodhpd2}
D^b(\mathcal{X})=\langle D^b(S), D^b(S), D^b(\mathcal{Z}'')\rangle
\end{equation}
with embeddings of the components given by Fourier Mukai functors with kernels $\O_{\mathcal{X}}, \mathcal{F}', \mathcal{E}''$. 
\end{thm}
%----------The Two Constructions Coincide in Characteristic 0-------------
\subsubsection{The Two Constructions Coincide in Characteristic 0}\label{coincide}
In this section, we assume $\mathrm{char}(k)=0$.

Previously we described two constructions for realizing the flat family of quintic del Pezzo surfaces $f: \mathcal{X}\to S$ as a linear section of a Grassmannian bundle over $S$. The first one uses the universal family $\mathcal{E}_2$ of the fine moduli space $\mathcal{M}_2(\mathcal{X}/S)$ and the second one uses the rank $2$ vector bundle $\mathcal{F}'$ defined by (\ref{normalseqfamily}). We prove in this section that they produce the same construction. More precisely, we will prove that 

(A) $\mathcal{F}'$ is the universal family of $\mathcal{M}_2(\mathcal{X}/S)$. Equivalently, $F'$ defined by (\ref{normalseq}) is a stable sheaf with Hilbert polynomial $h_2(t)$ (By Lemma \ref{semistable}(i), $F'_{\bar{k}}\cong F$ where $F$ is defined by (\ref{vb}).).

(B) In (\ref{normalseqfamily}), $(m_*\mathcal{I}(2))^*\cong f_*\mathcal{F}'$ and the map $ f^*(m_*\mathcal{I}(2))^*\to \mathcal{F}'$ is the evaluation map. Equivalently, the map $\O_X^{\oplus 5}\to F'$ in (\ref{normalseq}) is the evaluation map $H^0(X, F')\otimes_k \O_X\to F'$.

First note that as a linear section of $\Gr(2,5)$, a quintic del Pezzo surface $X$ with rational Gorenstein singularities is a projective l.c.i. scheme. Therefore, Hirzebruch-Riemann-Roch still holds for $X$ \cite[Corollary 18.3.1]{fint}. One computes the Hilbert polynomial with respect to $-K_X$ for a vector bundle $E$ on $X$ as follows:
\begin{equation}\label{hrr}
\begin{array}{rl}
h_E(t) &:=\chi(E(-tK_X))\\
&=\frac{K_X^2}{2}\mathrm{rk}(E)t^2+(\frac{K_X^2}{2}\mathrm{rk}(E)-c_1(E).K_X)t\\
&+\int_X\frac{c_1(E)^2-2c_2(E)-c_1(E).K_X}{2}+\int_X\frac{K_X^2+c_2(T_X)}{12}\mathrm{rk}(E)\\
\end{array}
\end{equation}
where $T_X=T_{\p^5}|_X-N_{X/\p^5}$ is the virtual tangent bundle. In particular, $1=\chi(\O_X)=h_{\O_X}(0)=\int_X \frac{K_X^2+c_2(T_X)}{12}$ implies that $\int_X c_2(T_X)=7$. 
\begin{lemma}
$h_{F'}(t)=h_F(t)=h_2(t)=5t^2+10t+5$.
\end{lemma}
\begin{proof}
$T_X=T_{\p^5}|_X-N_{X/\p^5}$ and the sequence (\ref{normalseq}) imply that $c_2(N_{X/\p^5})+c_2(T_X)+c_1(N_{X/\p^5})c_1(T_X)=c_2(T_{\p^5}|_X)$ and $c_2(N_{X/\p^5}(-2))+c_2(F')+c_1(N_{X/\p^5}(-2))c_1(F')=0$. Thus, $\int_X c_2(N_{X/\p^5})=43$ and $\int_X c_2(F')=2$. By the equation (\ref{hrr}), we have $h_{F'}(t)=5t^2+10t+5$.
\end{proof}

\cite[Proposition 4.2(i)]{hwgdp} indicates that there is a smooth elliptic curve $C\in |-K_X|$ on $X$ which does not meet the singular locus of $X$.

\begin{prop}
Let $M$ be a rank $2$ vector bundle on $X$ with determinant $\O_X(-K_X)$. If $M|_C$ is indecomposable on $C$, then $M$ is Gieseker stable with respect to $-K_X$.
\end{prop}
\begin{proof}
It suffices to prove the lemma when the base field $k$ is algebraically closed. Let $\pi: \wt{X}\to X$ be the minimal resolution. Then $C$ can be embedded into $\wt{X}$ and $C\in |-K_{\wt{X}}|$. Recall that $\pi^*F=\wt{F}$. Restricting the sequence (\ref{uex}) to $C$, one gets
\begin{equation}\label{uexc}
0\to \O_C(-K_{\wt{X}}-h)\to F|_C\to \O_C(h)\to 0.
\end{equation}
The short exact sequence
\[0\to \O_{\wt{X}}(-2h)\to \O_{\wt{X}}(-K_{\wt{X}}-2h)\to \O_C(-K_{\wt{X}}-2h)\to 0\]
induces the exact sequence $H^1(\wt{X}, \O_{\wt{X}}(-2h))=0\to H^1(\wt{X}, \O_{\wt{X}}(-K_{\wt{X}}-2h))=k\to H^1(C, \O_C(-K_{\wt{X}}-2h))=k$. Hence, the sequence (\ref{uexc}) is a nontrivial extension.

We claim that $F|_C$ is a stable rank $2$ bundle on $C$, that is, for any line bundle $L$ on $C$, $\deg(L)<\deg(F|_C)/\mathrm{rk}(F|_C)=5/2$. This is true because either $L$ is a subbundle of $\O_C(-K_{\wt{X}}-h)$ or a nontrivial subbundle of $\O_C(h)$ (since the extension is nontrivial). Thus, $\deg(L)\leqslant \deg(\O_C(-K_{\wt{X}}-h))=2$ or $\deg(L)<\deg{\O_C(h)}=3$. As a stable bundle, $F|_C$ is indecomposable on $C$.

By assumption, $h_M(t)=5t^2+10t+$ constant. Let $E$ be a rank $1$ saturated subsheaf of $M$, i.e. the quotient $G:=M/E$ is torsion free. By \cite[Proposition 1.1]{hstable}, $E$ is reflexive and thus $E|_C$ is a line bundle. Consider the short exact sequence
\[0\to E(K_X)\to E\to E|_C\to 0.\]
The first map is injective because $E$ is torsion free. Let $h_E(t)=\frac{5}{2}t^2+bt+c$. Then $\chi(E\otimes \O_C(-tK_X))=h_E(t)-h_E(t-1)=5t+b-\frac{5}{2}$. By Riemann-Roch, it is equal to $5t+\deg(E|_C)$. Thus, $b=\deg(E|_C)+\frac{5}{2}$. If $b<\frac{10}{2}=5$, then $h_E(t)<\frac{1}{2}h_{M}(t)$ and $M$ is stable.

Note that $\ker(E|_C\to M|_C)\cong \ker(G(K_X)\to G)\cong 0$ because $G$ is torsion free. We have a short exact sequence
\[0\to E|_C\to M|_C\to G|_C\to 0.\]
By \cite[Theorem 7]{ativbell}, an indecomposable vector bundle on a smooth elliptic curve is determined by its determinant when the rank and degree of the bundle are coprime. Since $\mathrm{rk}(M|_C)=2$ and $\deg(M|_C)=\deg(\O_C(-K_X))=5$ and the same is true for $F|_C$, we have $M|_C\cong F|_C$. In particular, $M|_C$ is stable on $C$ and thus $\deg(E|_C)<\frac{5}{2}$. The result follows.
\end{proof}
\begin{rem}
The proof works for any characteristic and it gives another proof for $F$ to be stable.
\end{rem}
\begin{lemma}\label{indec}
$F'|_C$ and $N_{X/\p^5}|_C$ are indecomposable on $C$.
\end{lemma}
\begin{proof} 
The surjection $\O_C^{\oplus 5}\to F'|_C$ realizes $C$ as $\p^4\cap\Gr(2, 5)$. Then on $\Gr(2,5)$, we have an exact Kozsul complex
\[0\to \O_{\Gr}(-5)\to\dots\to\O_{\Gr}(-1)^{\oplus 5}\to \O_{\Gr}\to\O_C\to 0. \]

Let $\R^*$ be the dual of the universal subbundle of $\Gr(2,5)$. Then $F'|_C\cong \R^*|_C$. Let $E:=\R\otimes\R^*$. Since $\mathrm{char}(k)=0$, $E\cong\Sym^2(\R^*)(-1)\oplus \O_{\Gr}$. By Bott's Theorem \cite[Corollary 4.1.9]{weygr}, $H^j(\Gr(2, 5), \Sym^2(\R^*)(-i))=0$ for all $j$ and $1\leqslant i\leqslant 6$. Then the Koszul complex implies that $\mathrm{End}(F'|_C)\cong H^0(C, E|_C)\cong H^0(C, \O_C)=k$. Hence, $F'|_C$ is indecomposable. A similar computation proves the claim for $N_{X/\p^5}|_C$.
\end{proof}
\begin{cor}
$F'_{\bar{k}}\cong F$. 
\end{cor}
\begin{proof}
Results above show that $F'$ is a stable sheaf with Hilbert polynomial $h_2(t)$ and the claim follows from Lemma \ref{semistable}(i).
\end{proof}
This concludes part (A) and now we will prove part (B). 
\begin{lemma}
The map $\O_X^{\oplus 5}\to F'$ in (\ref{normalseq}) is the evaluation map $H^0(X, F')\otimes_k \O_X\to F'$.
\end{lemma}
\begin{proof}
Since $h^0(X, F')=h^0(X_{\bar{k}}, F)=5$, it is equivalent to show that the induced map $H^0(X, \O_X^{\oplus 5})\to H^0(X, F')$ is an isomorphism. Assume the image of the map has dimension $l$. Then $\O_X^{\oplus 5}\to F'$ factors through the surjection $\O_X^{\oplus l}\to F'$ and $K:=\ker(\O_X^{\oplus l}\to F')$ is a direct summand of $N_{X/\p^5}(-2)$. Since $F'$ is not a trivial bundle, $l\geqslant 3$. If $l<5$, then $K$ is a proper direct summand and thus $N_{X/\p^5}(-2)$ is decomposable. But by Lemma \ref{indec}, this is impossible.
\end{proof}
\begin{thm}\label{identical}
Let $f:\mathcal{X}\to S$ be a flat family of quintic del Pezzo surfaces over $k$ as in Definition \ref{qdpfdef}. Let $\O(1)=\omega_{\mathcal{X}/S}^{-1}$ and $m: \p_S((f_*\O(1))^*)\to S$ be the projection. Let $\mathcal{I}$ be the ideal sheaf of the anticanonical embedding $\mathcal{X}\to \p_S((f_*\O(1))^*)$ and $N_{\mathcal{X}/\p_S}$ be the normal bundle. Let $\mathcal{F}'$ be the rank $2$ vector bundle defined by (\ref{normalseqfamily}). Then

(i) $\mathcal{F}'\cong \mathcal{E}_2$ is the universal family of $\M_2(\mathcal{X}/S)$ (up to the pull-back of a line bundle on $\M_2(\mathcal{X}/S)$).

(ii) $(m_*\mathcal{I}(2))^*\cong f_*\mathcal{F}'$ and the map $f^*(m_*\mathcal{I}(2))^*\to \mathcal{F}'$ is the evaluation map. In particular, we have a short exact sequence
\[0\to N_{\mathcal{X}/\p_S}(-2)\to f^*f_*\mathcal{E}_2\to \mathcal{E}_2\to 0\]

(iii) The linear section structures constructed in (\ref{first}) and (\ref{second}) are isomorphic. Hence, Theorem \ref{main2} and \ref{main3} are the same. In particular, $g''=g'=g: \mathcal{Z}''\cong \mathcal{Z}'\cong \mathcal{Z}\to S$.
\end{thm}
\begin{proof}
Note that by definition there is a natural map $(m_*\mathcal{I}(2))^*\to f_*\mathcal{F}'$. (i)(ii) follow from the results on  fibers of $f$ and (iii) is a consequence of (i)(ii).
\end{proof}
%-----------Appendix: Grassmannians in Arbitrary Characteristic---------
\appendix
\section{Grassmannians in Arbitrary Characteristic}
Let $k$ be an arbitrary field and $V$ be a $k$-vector space of dimension $n$. Let $\R$ be the universal subbundle of $\Gr(r, V)$ of rank $r$ and $\R^{\perp}$ be the kernel of the evaluation map $\O_{\Gr(r, V)}\otimes_k V^*\to \R^*$. We call  $\alpha=[\alpha_1, \dots,\alpha_n]$ a weight if all $\alpha_i\in\mathbb{Z}$. A weight $\alpha$ is \textit{dominant} if $\alpha_1\geqslant\dots\geqslant\alpha_n$ and is a \textit{partition} if in addition $\alpha_n\geqslant 0$. For a partition $\alpha$, it corresponds to a Young diagram with $\alpha_i$ boxes in the $i$-th row. Write $\alpha'$ for the Young diagram transposed to $\alpha$ and $|\alpha|=\sum_i \alpha_i$ for the \textit{degree}. Denote the Schur and Weyl functors by $L^{\alpha}, K^{\alpha}$ respectively and when $\alpha$ is a partition, they are defined by
\[L^\alpha V=\mathrm{im}(\bigotimes_i\bigwedge^{\alpha'_i}V\overset{a^{\vee}}{\longrightarrow}V^{\otimes |\alpha|}\overset{s}{\longrightarrow}\bigotimes_i\Sym^{\alpha_i} V),\]
\[K^\alpha V=\mathrm{im}(\bigotimes_i D^{\alpha_i}V\overset{s^\vee}{\longrightarrow} V^{\otimes |\alpha|}\overset{a}{\longrightarrow}\bigotimes_i\bigwedge^{\alpha'_i}V)\]
where $D^{\alpha_i} V$ is the divided power and $a,s$ are antisymmetrization and symmetrization maps respectively. In general, if $\alpha$ is only dominant, then $L^\alpha V$ is defined by $L^{\bar{\alpha}}V\otimes_k \bigwedge^{-\alpha_n} V^*$ where $\bar{\alpha}=[\alpha_1-\alpha_n,\dots, \alpha_{n-1}-\alpha_n,0]$ and $K^\alpha$ is defined similarly. \cite[Lemma 2.2]{blvdbgr} implies that for a dominant weight $\alpha$,
\[L^{\alpha_1,\dots,\alpha_n}V^*=(K^\alpha V)^*=L^{-\alpha_n,\dots,-\alpha_1}V.\]

In positive characteristic, Borel-Weil-Bott's theorem is only partially valid. Kempf vanishing theorem suggests that for dominant weights $\gamma=[\gamma_1,\dots,\gamma_r], \beta=[\beta_1,\dots,\beta_{n-r}]$, if $\gamma_r\geqslant \beta_1$, then
\begin{equation}
H^i(\Gr(r, V),L^{\gamma}\R^*\otimes L^\beta \R^\perp)=
\left\{
\begin{array}{ll}
L^{\gamma_1,\dots,\gamma_r, \beta_1,\dots,\beta_{n-r}} V^* & i=0\\
0, & i>0\\
\end{array}
\right.
\end{equation}

We will point out that in fact, the proof of Proposition 1.4 in \cite{blvdbgr} provides the following algorithm for some additional vanishing of cohomologies:
\begin{prop}
Let $\p=\p_k^{n-r}$ and $\mathcal{N}$ be the kernel of the evaluation map $\O_{\p}^{n-r+1}\to\O_{\p}(1)$, i.e. $\mathcal{N}=\Omega_\p(1)$. If all cohomologies $H^\bullet(\p, \O_\p(\gamma_r)\otimes L^\beta\mathcal{N})$ vanish, then $H^i(\Gr(r, V),L^{\gamma}\R^*\otimes L^\beta \R^\perp)=0$ for all $i$.
\end{prop} 
In particular, it indicates that \cite[Lemma 3.2(a)]{kaphom} still holds in arbitrary characteristic. But Lemma 3.2(b) in \textit{loc.cit.} may not be true in positive characteristic.
\begin{cor}
Let $\gamma=[\gamma_1,\dots, \gamma_r]$. Suppose $\gamma_1\geqslant\dots\gamma_r\geqslant -(n-r)$. Then 
\[
H^i(\Gr(r, V),L^{\gamma}\R^*)=
\left\{
\begin{array}{ll}
0, & i>0\\
0, & i=0, \gamma_r<0\\
L^{\gamma}V^*, & i=0, \gamma_r\geqslant 0\\
\end{array}
\right.
\]
\end{cor}

From now on, we will focus on the case $\dim_k V=5$. In arbitrary characteristic, it is unclear whether the collection
\begin{equation}\label{excep}
\langle\O,\R^*, \O(1),\R^*(1),\O(2),\R^*(2),\O(3), \R^*(3),\O(4),\R^*(4)\rangle
\end{equation}
is semiorthogonal because it requires Lemma 3.2(b) in \cite{kaphom}. We will prove below that it is actually a full exceptional collection of $\Gr(2,5)$ in large characteristic by producing the collection from Kapranov's collection via right mutations.

Assume $\mathrm{char}(k)=0$ or $\geqslant 5$. From \cite[Lemma 7.7]{blvdbgr}, one has 
\begin{equation}\label{ds}
(\R^*)^{\otimes 2}=\Sym^2\R^*\oplus \O(1),\quad (\R^*)^{\otimes 3}=\Sym^3\R^*\oplus (\R^*(1))^{\oplus 2}
\end{equation}
and for each decomposition, there are no RHom between their direct summands. In particular, we have 
\begin{equation}\label{ds1}
\R^*\otimes\Sym^2\R^*=\Sym^3\R^*\oplus \R^*(1).
\end{equation}
There is a semiorthogonal decomposition \cite[Corollary 7.8]{blvdbgr}
\[
\begin{array}{l}
D^b(\Gr(2, V))\\
=\langle\O,\R^*,\O(1), \Sym^2\R^*, \R^*(1), \Sym^3\R^*, \O(2), (\Sym^2\R^*)(1), \R^*(2), \O(3)\rangle.
\end{array}
\]
We perform the following right mutations:

(i) Move $\Sym^3\R^*$ to the rightmost and it becomes $\O(4)$;

(ii) Move $\Sym^2\R^*$ towards right past $\O(3)$ and it becomes $\R^*(3)$;

(iii) Move $(\Sym^2\R^*)(1)$ to the rightmost and it becomes $\R^*(4)$.

Note that we have $\R=\R^*(-1), \bigwedge^2 (\R^{\perp})^*=\R^{\perp}(1)$ and a short exact sequence
\begin{equation}\label{cex}
0\to \R^\perp\to V^*\otimes_k\O\to \R^*\to 0
\end{equation}
which induces filtrations
\begin{equation}\label{fil1}
0\to N\to\bigwedge^2 V^*\otimes_k \O\to\bigwedge^2\R^*=\O(1)\to 0,
\end{equation}
\begin{equation}\label{fil2}
0\to \bigwedge^2\R^\perp=(\R^\perp)^*(-1)\to N\to \R^\perp\otimes\R^*\to 0
\end{equation}
and a short exact sequence
\begin{equation}\label{dcex}
0\to \R\to V\otimes_k \O\to (\R^{\perp})^*\to 0
\end{equation}
which induces filtrations
\begin{equation}\label{dfil1}
0\to M\to \bigwedge^2 V\otimes_k\O\to \bigwedge^2(\R^{\perp})^*=\R^{\perp}(1)\to 0,
\end{equation}
\begin{equation}\label{dfil2}
0\to \bigwedge^2\R=\O(-1)\to M\to \R\otimes(\R^{\perp})^*\to 0.
\end{equation}

Below we give details for each step of right mutations. All diagrams are commutative with exact rows and columns. We denote coevaluation maps by coev.

(i.1): From direct computations, $\RHom(\Sym^3\R^*, \O(2))=0$, $\RHom(\Sym^3\R^*, (\Sym^2\R^*)(1)) = H^0(\Gr(2,V), \R^*) = V^*$. Hence, right mutations of the triple $(\Sym^3\R^*, \O(2), (\Sym^2\R^*)(1))$ is $(\O(2), (\Sym^2\R^*)(1), K_1)$ where $K_1$ is described as follows:

\begin{equation}\label{K1}
\begin{tikzcd}[column sep=tiny]
& & 0\arrow{d} & 0\arrow{d}\\
0\arrow{r} & \Sym^3\R^*\arrow{r}\arrow[equal]{d} & \R\otimes(\Sym^2\R^*)(1)=\R^*\otimes\Sym^2\R^*\arrow{r}\arrow{d} & \R^*(1)\arrow{r}\arrow{d} & 0\\
0\arrow{r} & \Sym^3\R^*\arrow{r}{\text{coev}} & V\otimes (\Sym^2\R^*)(1)\arrow{r}\arrow{d} & K_1\arrow{r}\arrow{d} & 0\\
& & (\R^{\perp})^*\otimes(\Sym^2\R^*)(1) \arrow[equal]{r}\arrow{d} & (\R^{\perp})^*\otimes(\Sym^2\R^*)(1)\arrow{d}\\
& & 0 & 0\\
\end{tikzcd}
\end{equation}
The first row comes from the decomposition (\ref{ds1}). The middle row is the sequence defining $K_1$. The middle column is (\ref{dcex}) tensoring with $(\Sym^2\R^*)(1)$. 

(i.2) The middle row of (\ref{K1}) implies $\RHom(K_1, \R^*(2))=\bigwedge^2V^*$. Right mutation of the pair $(K_1, \R^*(2))$ is $(\R^*(2), K_2)$ with $K_2$ described as follows:

\begin{equation}\label{K2}
\begin{tikzcd}
& 0\arrow{d} & 0\arrow{d}\\
& K_1\arrow[equal]{r}\arrow{d} & K_1\arrow{d}{\text{coev}}\\
0\arrow{r} & M\otimes \R^*(2)\arrow{r}\arrow{d} & \bigwedge^2 V\otimes\R^*(2)\arrow{r}\arrow{d} & \R^{\perp}\otimes\R^*(3)\arrow{r}\arrow[equal]{d} & 0\\
0\arrow{r} & (\R^{\perp})^*(2)\arrow{r}\arrow{d} & K_2\arrow{r}\arrow{d} & \R^{\perp}\otimes\R^*(3)\arrow{r} & 0\\
& 0 & 0\\
\end{tikzcd}
\end{equation}
The middle column is the sequence defining $K_2$. The middle row is (\ref{dfil1}) tensoring with $\R^*(2)$. The first column is the middle column below:

\begin{equation}
\begin{tikzcd}
& & 0\arrow{d} & 0\arrow{d}\\
0\arrow{r} & \R^*(1)\arrow{r}\arrow[equal]{d} & K_1\arrow{r}\arrow{d} & (\R^\perp)^*(1)\otimes\Sym^2\R^*\arrow{r}\arrow{d} & 0\\
0\arrow{r} & \R^*(1)\arrow{r} & M\otimes\R^*(2)\arrow{r}\arrow{d} & (\R^\perp)^*(1)\otimes(\R^*)^{\otimes 2}\arrow{r}\arrow{d} & 0\\
& & (\R^\perp)^*(2)\arrow[equal]{r}\arrow{d} & (\R^\perp)^*(2)\arrow{d}\\
& & 0 & 0\\
\end{tikzcd}
\end{equation}
The first row is the last column from (\ref{K1}). The second row is (\ref{dfil2}) tensoring with $\R^*(2)$. The last column comes from the decomposition of $(\R^*)^{\otimes 2}$ tensoring with $(\R^\perp)^*(1)$.

(i.3) Comparing the last row of (\ref{K2}) with (\ref{fil2}), one has $K_2=N(3)$. The sequence (\ref{fil1}) implies that the right mutation of $(K_2, \O(3))$ is $(\O(3), \O(4))$. 

(ii.1) $\RHom(\Sym^2\R^*,\R^*(1))=H^0(\Gr(2,V),\R^*)=V^*$. Right mutation of $(\Sym^2\R^*, \R^*(1))$ is $(\R^*(1), L)$ with $L$ described as follows:

\begin{equation}\label{L}
\begin{tikzcd}
& & 0\arrow{d} & 0\arrow{d}\\
0\arrow{r} & \Sym^2\R^*\arrow{r}\arrow[equal]{d} & \R\otimes\R^*(1)=(\R^*)^{\otimes 2}\arrow{r}\arrow{d} & \O(1)\arrow{r}\arrow{d} & 0\\
0\arrow{r} & \Sym^2\R^*\arrow{r}{\text{coev}} & V\otimes \R^*(1)\arrow{r}\arrow{d} & L\arrow{r}\arrow{d} & 0\\
& & (\R^\perp)^*\otimes\R^*(1)\arrow[equal]{r}\arrow{d} & \R\otimes(\R^\perp)^*(2)\arrow{d}\\
& & 0 & 0\\
\end{tikzcd}
\end{equation}
The first row comes from the decomposition of $(\R^*)^{\otimes 2}$. The middle row is the sequence defining $L$. The middle column is (\ref{dcex}) tensoring with $\R^*(1)$. 

(ii.2) Comparing the last column of (\ref{L}) with (\ref{dfil2}), one has $L=M(2)$ and $\RHom(L,\O(2))=\bigwedge^2V^*$. The sequence (\ref{dfil1}) tensoring with $\O(2)$ implies that the right mutation of $(L, \O(2))$ is $(\O(2), \R^\perp(3))$.

(ii.3) $\RHom(\R^\perp(3), (\Sym^2\R^*)(1))=\RHom(\R^\perp(3),\R^*(2))=0$. The sequence (\ref{cex}) tensoring with $\O(3)$ implies that the right mutation of $(\R^\perp(3), \O(3))$ is $(\O(3),\R^*(3))$.

(iii) follows from (ii.1-3).

One should note that it is important that we have direct sum decompositions (\ref{ds})(\ref{ds1}), namely the factors of the natural filtrations of $(\R^*)^{\otimes 2}, (\R^*)^{\otimes 3}$ and $\R^*\otimes \Sym^2\R^*$ are in fact direct summands. On one hand, it enables us to make cohomological computations involving $\Sym^3\R^*$. On the other hand, the argument above uses the short exact sequences coming from the splitting maps of the filtrations several times. Thus, one would expect that the collection (\ref{excep}) may not be semiorthogonal in characteristic $2$ or $3$. We summarize the result below.
\begin{prop}\label{lefschetzgr}
Let $k$ be a field and $\mathrm{char}(k)\neq 2,3$. Then $\Gr(2,5)$ over $k$ has a semiorthogonal decomposition (of rectangular Lefschetz type)
\[D^b(\Gr(2,5))=\langle\O,\R^*, \O(1),\R^*(1),\O(2),\R^*(2),\O(3), \R^*(3),\O(4),\R^*(4)\rangle\]
where $\R$ is the universal subbundle of rank $2$.
\end{prop}

\bibliography{dp5.bib}
\bibliographystyle{abbrv}
\end{document}